\documentclass[a4paper,11pt]{article}
\usepackage{a4wide,url}
\usepackage{amssymb,amsmath,amsthm,longtable}
\usepackage{fullpage}
\usepackage{color}
\usepackage{graphicx}
\usepackage{floatflt}
\usepackage{paralist}
\usepackage[font = small,labelfont=bf,format=hang,justification=raggedright]{caption}
\usepackage{verbatim}
\usepackage{lipsum}

\usepackage{color}
\usepackage{graphicx}
\usepackage{floatflt}
\usepackage{paralist}
\usepackage[font = small,labelfont=bf,format=hang,justification=raggedright]{caption}
\usepackage{verbatim}
\usepackage{lipsum}

\usepackage{caption}
\usepackage{subcaption}
\newcommand{\bitem}{\begin{itemize}}
\newcommand{\eitem}{\end{itemize}}
\newcommand{\beq}{\begin{equation}}
\newcommand{\eeq}{\end{equation}}
\newcommand{\ip}[2]{\left\langle#1,#2\right\rangle}

\newcommand{\absip}[2]{| \langle#1,#2\rangle |}
\newcommand{\norm}[1]{\|#1\|}
\newcommand{\cE}{\mathcal{E}}

\newcommand{\cR}{\mathcal{R}}
\newcommand{\cH}{\mathcal{H}}
\def\diag{{\text{\rm diag}}}
\def\epsilon{\varepsilon}
\def\phi{\varphi}

\newcommand{\Z}{\mathbb{Z}}
\newcommand{\N}{\mathbb{N}}

\newcommand{\R}{\mathbb{R}}

\newcommand{\C}{\mathbb{C}}

\newcommand{\intRtwo}{\int \limits_{\R^2}}

\newtheorem{theorem}{Theorem}[section]
\newtheorem{remark}{Remark}[section]

\newtheorem{definition}[theorem]{Definition}

\newtheorem{corollary}[theorem]{Corollary}

\DeclareMathOperator{\suppp}{supp \,}
\DeclareMathOperator*{\argmin}{arg\,min}

\newcommand{\hl}[1]{{\color{black}{#1}}} 

\title{Regularization and Numerical Solution of the Inverse Scattering Problem Using Shearlet Frames}

\author{Gitta Kutyniok\footnotemark[1] \and Volker Mehrmann\footnotemark[1] \and Philipp Petersen\footnotemark[1]}

\begin{document}

\maketitle

\begin{abstract}
Regularization techniques for the numerical solution of inverse scattering problems in two space dimensions are discussed. Assuming that the boundary of a scatterer is its most prominent feature, we exploit as model the class of cartoon-like functions. Since functions in this class are asymptotically optimally sparsely approximated by shearlet frames, we consider shearlets as a means for regularization. We analyze two approaches, namely solvers for the nonlinear problem and for the linearized problem obtained by the Born approximation. As example for the first class we study the acoustic inverse scattering problem, and for the second class, the inverse scattering problem of the Schr\"{o}dinger equation. Whereas our emphasis for the linearized problem is more on the theoretical side due to the standardness of associated
solvers, we provide numerical examples for the nonlinear problem that highlight the effectiveness of our algorithmic approach. 
\end{abstract}
\noindent {\bf Keywords.} Helmholtz equation, Inverse medium scattering, Regularization, Schr\"odinger equation, Shearlets, Sparse approximation.

\noindent {\bf AMS subject classification.} 34L25, 35P25, 42C40, 42C15, 65J22, 65T60, 76B15, 78A46

\renewcommand{\thefootnote}{\fnsymbol{footnote}}

\footnotetext[1]{Department of Mathematics, Technische Universit\"at Berlin, 10623 Berlin, Germany;
\texttt{Email-Addresses: $\{$kutyniok,mehrmann,petersen$\}$@math.tu-berlin.de}.}
\section{Introduction}
\label{sec:introduction}

The scattering problem analyzes how incident waves, radiation, or particles, which are transmitted in a medium, are
scattered at inhomogeneities of this medium. The associated inverse problem aims to determine characteristics of
the inhomogeneities from the asymptotic behavior of such scattered waves. This problem appears in various
flavors in different application areas, e.g. non-destructive testing, ultrasound tomography,
and echolocation. For an overview of the problem and recent developments, we refer to the survey article
\cite{CCMScatteringSurvey2000}.

Various numerical methods have been proposed for the solution of inverse scattering problems.
A very common approach to solve a nonlinear inverse scattering problem are  \emph{fix-point iterations}, which produce a sequence of \emph{linear} inverse scattering problems with solutions that converge, under some suitable assumptions, to a solution of the nonlinear problem. One such
approximation technique is the \emph{Born approximation}, see e.g. \cite{BaoTr2010,MosS}. However, one drawback of this class of approaches is the fact that it requires the solution of a linear inverse scattering problem in every
iteration step, which is typically again an \emph{ill-posed} problem that is hard to solve in the presence of noisy data or data with linearization errors.
On the other hand, the nonlinear problem can be solved directly by an optimization approach, see \cite{ColK}. One particular such technique, see \cite{LKKLpRegularization2013}, tackles the nonlinear problem by minimizing a \emph{Tikhonov functional} with a suitably chosen regularization term. The success of such an approach depends heavily on how properties of the solution are encoded in the
regularization term. This, however, requires typically that a priori knowledge about characteristics of
the solution is available.

We discuss both these approaches and combine them with a sparsity based methodology which makes use of representing the scatterer in a sparse way, as it has been suggested in several other areas of inverse problems.
This methodology is based on the hypothesis that most types of data indeed admit a sparse approximation by
a suitably chosen basis or frame,
see Subsection \ref{subsec:frames},
and today this is a well-accepted paradigm. Generally speaking, knowledge of a sparsifying
basis or frame, appropriately applied, allows precise and stable reconstruction from very few and even noisy measurements. One prominent way to infuse such knowledge is by a regularization term such as in a Tikhonov functional. Indeed, in \cite{LKKLpRegularization2013}, it is assumed that the to-be-detected objects are sparse in the sense of small support, which is then encoded by using an $L^p$-norm for $p$ close to $1$ as regularization term, thereby promoting sparsity.

In this paper, we also aim to utilize sparsity to solve inverse medium scattering problems, but follow a different
path. The key idea of our new approach is to generate a model for a large class of natural structures and an associated representation system, which provides asymptotically optimal sparse approximation of elements of this model class. We use this approach for solving the nonlinear as well as a linearized inverse scattering problem. As problem cases we consider the \emph{acoustic inverse medium scattering problem} and the \emph{inverse scattering problem of the Schr\"{o}dinger equation}.
\subsection{Modeling of the Scatterer}

Typically, a scatterer is a natural structure, which distinguishes itself from the surrounding medium
by a change in density. In the 2D setting, this inhomogeneity can be regarded as a curve with, presumably, certain regularity properties. The interior as well as the exterior of this curve  is usually assumed to be homogeneous.

In the area of imaging sciences, the class of \emph{cartoon-like functions} \cite{DCartoonLikeImages2001} is frequently used as model for images governed by
anisotropic structures such as edges. 
Roughly speaking, a cartoon-like function is a compactly supported function which is a twice continuously differentiable function, apart from a piecewise $C^2$ discontinuity curve, see Definition \ref{def:CartoonLikeFunction} below.
This cartoon-like model is well-suited  for  many inverse scattering problems, where the discontinuity
curve 
models the boundary of a homogeneous domain. In some physical applications, one may debate this regularity of the curve as well as the homogeneity of the domains, but a certain smoothness on small pieces of the boundary seems a realistic scenario.
\subsection{Directional Representation Systems}

Having agreed on a model, one needs a suitably adapted representation system which ideally provides asymptotically optimal sparse approximations of cartoon-like functions in the sense of the decay of the $L^2$-error of best $N$-term approximation. Such a system can then be used for the regularization term of a Tikhonov functional.

The first (directional) representation system which achieved asymptotic optimality were \emph{curvelets} introduced
in \cite{CD2004}. In fact, in \cite{CandD2002CurveletsInIllPosedProblems} curvelets are used to regularize linear
ill-posed problems. This is done under the premise that the solution of the inverse problem exhibits edges, which tend to get smoothed out in a regularization procedure, while curvelets as a system adapted to edges overcomes
this obstacle. However, on the practical side, curvelets suffer from the fact that
often a faithful numerical realization of the associated transform is difficult.

In  \cite{GKL2006} shearlet systems were introduced,  which similarly achieve the desired optimal sparse approximation rate \cite{KLcmptShearSparse2011}, but in addition allow a unified treatment of the continuum and digital realm \cite{shearlab}. As curvelets, shearlets are mainly designed for image processing applications, in
which they are also used for different inverse problems such as separation of morphologically distinct components
\cite{Donoho2010c,KLImageSeparationWaveShear2012}, recovery of missing data \cite{GK15,king2014analysis},
or
reconstruction from the Radon transform \cite{ColEGL2010ShearletsRadon}. Furthermore, in contrast to curvelets, compactly supported shearlet frames for high spatial localization are available \cite{KGLConstrCmptShear2012}, see  \cite{KL2012} for a survey. 

In view of this discussion, shearlet frames seem a good candidate as a regularizer for inverse scattering problems, and in fact this will be key to our approach.

\subsection{Inverse Scattering Problems}
We examine two conceptually different approaches to numerically solve inverse scattering problems. More precisely, we study a method to directly tackle the
nonlinear problem as well as a linearization approach. As problem cases we focus on two types of inverse scattering problems, see e.g. \cite{ColK}, which are
the \emph{acoustic inverse scattering problem}, and the \emph{inverse scattering problem of the Schr\"{o}dinger equation}. For the second, we analyze the
strategy to linearize the inverse scattering problem by means of the Born approximation.

The \emph{acoustic inverse scattering problem} aims to reconstruct a contrast function which encodes the scatterer by
emitting an acoustic wave and measuring the scattered waves. Common application areas are radar, sonar,
and geophysical exploration, see e.g. \cite{ColK} for a survey of applications.

The minimization of a suitable Tikhonov functional is a common approach to directly solve this nonlinear inverse
problem. In 
\cite{LKKLpRegularization2013} a sparsity-based regularization term is introduced which uses the $L^p$-norm with $p$ close to $1$ directly on the function to-be-recovered. This regularization scheme is very successful when the object under consideration has small support.

Following our methodological concept, and assuming that cartoon-like functions are an appropriate model for the
scatterer, we instead choose as regularization term the $\ell_p$-norm of the associated shearlet coefficient sequence with $p$ larger or equal to $1$.
After giving the theoretical background in Section \ref{sec:acoustic}, we present numerical experiments that compare our approach to that of \cite{LKKLpRegularization2013}, see Subsection~\ref{subsec:exp_results}. These examples show convincing results, both in terms of the reconstruction error and the number of iterations. In particular, it is demonstrated that edges of the scatterer are recovered with high accuracy.


The \emph{inverse scattering problem of the Schr\"{o}dinger equation} aims to determine a quantum mechanical scattering potential from measurements of backscattering data.
A prominent method to linearize this inverse scattering problem is by means of the Born approximation.

Modeling the scatterer by cartoon-like functions, shearlets can be used again as a regularizer, provided that the transition from the nonlinear towards the linear problem does not influence the
fact that the solution belongs to the class of cartoon-like functions. It has been shown in \cite{OlaPS2001,Serov2013} that certain singularities of the scatterer can still be found in the
solution of the associated linearized problem.
However, all these results require a global regularity of the scatterer to describe the regularity of the inverse Born approximation.
On the other hand, in the case of cartoon-like functions we have strong local but poor global regularity and therefore the results of \cite{OlaPS2001,Serov2013} can not be applied to our situation. To provide a theoretical basis for the application of shearlet frames, we prove that indeed the Born approximation
to the Schr\"odinger equation gives rise to a scattering problem that exhibits sharp edges in the
solution of the linearized problem. In particular, we show that the cartoon model is almost invariant under the
linearization process, see Theorem \ref{thm:MainSchroedinger} and Corollary \ref{cor:cartoon}. This implies that the sparsity structure of the model stays untouched by the linearization. These results
then provide the theoretical justification that shearlet systems can be used as regularization for the numerical solution of the associated linearized problems.

\subsection{Outline of the Paper}

The paper is organized as follows. The precise definition of shearlet systems, their frame properties, and
their sparse approximation properties for cartoon-like functions are summarized in Section \ref{sec:shearlets}.
Section \ref{sec:acoustic} is devoted to the nonlinear scattering problem. We first describe the direct
and associated inverse problem, followed by the introduction of our new approach to regularize the inverse
scattering problem using the shearlet transform in Subsection \ref{subsec:reg_frames}. In Section \ref{sec:numEx} these methods are then compared numerically to other approaches.

The scattering problem of the Schr\"{o}dinger equation, which will be solved by a linearization, is then introduced and
studied in Section \ref{sec:electro} with Theorem~\ref{thm:MainSchroedinger} being the main result on
local regularity of the inverse Born approximation. Corollary~\ref{cor:cartoon} analyzes the situation
of using the cartoon-like model as scatterer. In Subsection \ref{subsec:Schr_Num} we describe possible numerical algorithms that use Corollary~\ref{cor:cartoon} as well as the effect on real world problems.

\section{Shearlet Systems}
\label{sec:shearlets}

In this section we provide a precise definition of shearlet frames and recall their sparse approximation properties,
see \cite{KL2012} for a survey on shearlets and \cite{Christensen2003a} for a survey on frames.

\subsection{Review of Frame Theory}
\label{subsec:frames}

A \emph{frame} generalizes the notion of orthonormal bases by only requiring a norm equivalence between the Hilbert space norm of a vector and the $\ell_2$-norm of the associated sequence of coefficients. To be more precise, given a Hilbert space $\cH$ and an index set $I$, then a
system $\{\varphi_i\}_{i \in I} \subset \cH$,  is called a \emph{frame} for $\cH$, if there exist constants $0< \alpha_{1} \leq \alpha_2 <\infty$
such that
\[
\alpha_1\|f\|^2 \leq \sum_{i \in I} \absip{f}{\varphi_i}^2 \leq \alpha_2\|f\|^2 \quad \mbox{for all } f \in \cH.
\]
The constants $\alpha_1$, $\alpha_2$ are referred to as the \emph{lower} and \emph{upper frame bound}, respectively. If $\alpha_1=\alpha_2$ is possible, then the frame is called \emph{tight}.

Each frame $\Phi:= \{\varphi_i\}_{i \in I} \subset \cH$ is associated with an \emph{analysis operator} $T_\Phi$ defined by
\[
T_\Phi : \cH \to \ell^2(I), \quad T_\Phi(f) = (\ip{f}{\varphi_i})_{i \in I},
\]
which decomposes a function into its \emph{frame coefficients}. The adjoint of $T_\Phi$ is called \emph{synthesis operator} and is given by
\[
T_\Phi^* : \ell^2(I) \to \cH, \quad T_\Phi^*((c_i)_{i \in I}) = \sum_{i \in I} c_i \varphi_i.
\]
Finally, the \emph{frame operator} is defined by $S_\Phi := T_\Phi^*T_\Phi$. The operator $S_\Phi$, which can be shown to be self-adjoint and invertible, see e.g. \cite{Christensen2003a}, allows both a reconstruction of $f$ (given its frame coefficients) and an expansion of $f$ in terms of the frame elements, i.e.,
\[
f = \sum_{i \in I} \ip{f}{\varphi_i} S_\Phi^{-1} \varphi_i = \sum_{i \in I} \ip{f}{S_\Phi^{-1} \varphi_i} \varphi_i \quad \mbox{for all } f \in \cH.
\]
Hence, although $\Phi$ does not constitute a basis, there exists a reconstruction formula using the system $\{\tilde{\varphi}_i\}_{i \in I}:= \{S_\Phi^{-1} \varphi_i\}_{i \in I}$, which can actually be shown to also form a frame, the \emph{so-called canonical dual frame}.

As for efficient expansions of a function $f \in \cH$ in terms of $\Phi$ we can identify with $(\ip{f}{\tilde{\varphi}_i})_{i \in I}$ one explicit coefficient sequence. However, this is typically by far not the `best' possible coefficient sequence in the sense of rapid decay in absolute value. Since one has better control over the sequence $(\ip{f}{\varphi_i})_{i \in I}$ of frame coefficients, it is often advantageous to instead consider the expansion $f = \sum_{i \in I} c_i \tilde{\varphi}_i$ of $f$ in terms of the canonical dual frame. The reason is that, if fast decay of the frame coefficients can be shown, then this form provides an efficient expansion of $f$.

\subsection{Shearlet systems and Frame Properties}

Shearlet systems are designed to asymptotically optimal encode geometric features in two space dimensions.
For their construction we define for $ j, k \in \Z$ the matrices
\[
A_j = \begin{bmatrix} 2^j & 0 \\ 0 & 2^{\frac{j}{2}} \end{bmatrix}, \quad \text{ and } S_k = \begin{bmatrix} 1 & k \\ 0 & 1 \end{bmatrix}.
\]
The former is called \emph{parabolic scaling matrix} and ensures that the elements of the shearlet system have an essential support of size $2^{-j} \times 2^{-\frac{j}{2}}$ following the
\emph{parabolic scaling law} `\emph{width} $\approx$ \emph{length}$^2$'. The matrix $S_k$ is called \emph{shearing matrix}, which in contrast to rotation matrices used in the construction of curvelets \cite{CD2004}, leave the digital grid $\Z^2$ invariant, and ensure the possibility of a faithful numerical realization.
The formal definition of a shearlet system as it was defined in \cite{KGLConstrCmptShear2012} is as follows.
\begin{definition}\label{def:shearletSystem}
Let $\phi, \psi, \tilde{\psi} \in L^2(\R^2)$, $c= [c_1,c_2]^T \in \R^2$ with $c_1,c_2>0$. Then the \emph{(cone-adapted) shearlet
system} is defined by
\[
 \mathcal{SH}(\phi, \psi, \tilde{\psi}, c) = \Phi(\phi, c_1) \cup \Psi(\psi, c) \cup \tilde{\Psi}(\tilde{\psi}, c),
\]
where
\begin{eqnarray*}
\Phi(\psi, c_1) &=& \left \{ \phi(\cdot - c_1 m) :m\in \Z^2 \right\},\\
\Psi(\psi, c) &=& \left\{ \psi_{j,k,m} = 2^{\frac{3j}{4}}\psi(S_k A_{j}\cdot -  M_c m ): j\in \N_0, |k| \leq   2^{\left\lceil\frac{j}{2}\right\rceil}, m\in \Z^2 \right\},\\
\tilde{\Psi}(\tilde{\psi}, c) &=& \left\{ \tilde{\psi}_{j,k,m} = 2^{\frac{3j}{4}}\tilde{\psi}(S_k^T \tilde{A}_{j}\cdot -  M_{\tilde{c}} m ): j\in \N_0, |k| \leq
2^{\left\lceil\frac{j}{2}\right\rceil}, m\in \Z^2 \right\},
\end{eqnarray*}
with $M_c:= \diag( c_1, c_2 )$, $M_{\tilde{c}} = \diag( c_2 , c_1 )$, and $\tilde{A}_{2^j} = \diag(2^{\frac{j}{2}},2^{j})$.
\end{definition}
One possibility to obtain a frame is to choose $\phi, \psi \in L^2(\R^2)$ such that, for $\alpha>\gamma>3$,
 \begin{eqnarray*}
 |\widehat{\phi}(\xi_1, \xi_2)| &\leq& C_1 \min\{ 1, |\xi_1|^{-\gamma}\} \min \{1, |\xi_2|^{-\gamma}\} \mbox{ and }\\
 |\widehat{\psi}(\xi_1, \xi_2)| &\leq& C_2 \min\{1,|\xi_1|^\alpha\}\min\{1, |\xi_1|^{-\gamma}\} \min \{1, |\xi_2|^{-\gamma}\},
\end{eqnarray*}
where $\widehat{g}$ denotes the Fourier transform of $g\in L^2(\R^2)$, and $\tilde{\psi}(x_1, x_2) := \psi(x_2, x_1)$. In this situation it has been established in \cite{KGLConstrCmptShear2012} that there exists a sampling vector $c= [c_1, c_2]^T\in \R^2$, $c_1,c_2>0$ such that $\mathcal{SH}(\phi, \psi, \tilde{\psi}; c)$
forms a frame for $L^2(\R^2)$. One special case are compactly supported shearlet frames.

%
%
Faithful implementations of shearlet frames and the associated analysis operators
are available at \url{www.ShearLab.org}, see also \cite{shearlab}.

\subsection{Sparse Approximation}
\label{subsec:sparseapprox}
Shearlets have well-analyzed approximation properties, in particular, for cartoon-like functions as initially introduced in \cite{DCartoonLikeImages2001}. Denoting by $\chi_D \in L^2(\R^2)$ the characteristic function on a bounded, measurable set $D \subset \R^2$, we have the following definition.
\begin{definition} \label{def:CartoonLikeFunction}
The class $\mathcal{E}^2(\R^2)$ of \emph{cartoon-like functions} is the set of functions $f:\R^2 \to \C$ of the form
\[
 f = f_0 + f_1 \chi_D,
\]
where $D \subset [0,1]^2$ is a set with $\partial D$ being a closed $C^2$-curve with bounded curvature and $f_i \in C^2(\R^2)$
are functions with support $\suppp f_i \subset [0,1]^2$ as well as $\|f_i\|_{C^2}\leq 1$ for $i = 0,1$.
\end{definition}
%
%
We measure the approximation quality of shearlets with respect to the
cartoon model by the decay of the $L^2$-error of best $N$-term approximation. Recall that for a general representation system
$\{\psi_i\}_{i\in I} \subset \mathcal{H}$ and $f\in \mathcal{H}$, the \emph{best $N$-term approximation} is defined as
\[
f_N = \argmin \limits_{\substack{\Lambda \subset \N, |\Lambda| = N,\\ \tilde{f}_N = \sum \limits_{i \in \Lambda} c_i \psi_i}} \|f- \tilde{f}_N \|.
\]
In contrast to the situation of orthonormal bases, if $\{\psi_i\}_{i\in I}$ forms a frame or even a tight frame, it is not clear at all how the set $\Lambda$ has to be chosen. Therefore, often the \emph{best $N$-term approximation} is substituted by the $N$-term approximation using the $N$ largest coefficients.

To be able to claim \emph{asymptotic optimality} of a sparse approximation, one requires a benchmark result. In
\cite{DCartoonLikeImages2001} it was shown that for an arbitrary representation system $\{\psi_i\}_{i\in I} \subset L^2(\R^2)$, the minimally achievable asymptotic approximation error for $f \in \mathcal{E}^2(\R^2)$ is
\[
\|f-f_N\|_2^2 = O(N^{-2}) \quad \mbox{as } N \to \infty,
\]
provided that only polynomial depth search is used to compute the approximation. The condition on the polynomial depth search means, that the $i$-th term in the expansion can be chosen in accordance with a selection rule
$\sigma(i,f)$, which obeys $\sigma(i,f) \leq \pi(i)$
for a fixed polynomial $\pi(i)$, see also \cite{DCartoonLikeImages2001}.

In the above definition of asymptotic optimality, we made use of the  Landau symbol $O(f(a))$, which for a function $f$  describes the asymptotic convergence behavior as $a \to 0$ for the set of functions $g$ such that $\limsup_{x\to a}$ $\frac{g(x)}{f(x)} < \infty$.

Shearlets achieve this asymptotically optimal rate up to a $\log$-factor as the following result shows.
\begin{theorem}[\cite{KLcmptShearSparse2011}]
\label{thm:ShearletOptimallySparseApproximation}
Let $\phi, \psi, \tilde{\psi} \subset L^2(\R^2)$ be compactly supported, and assume that the shearlet system
$\mathcal{SH}(\phi, \psi, \tilde{\psi}, c)$ forms a frame for $L^2(\R^2)$. Furthermore, assume that, for all $\xi =
[\xi_1, \xi_2]^T \in \R^2$, the function $\psi$ satisfies
\begin{eqnarray*}
|\widehat{\psi}(\xi)| &\leq& C \min\{1, |\xi_1|^{\delta}\} \min \{1, |\xi_1|^{-\gamma}\} \min \{1, |\xi_2|^{-\gamma}\},\\
\left |\frac{\partial }{\partial \xi_2}\widehat{\psi}(\xi)\right | &\leq& |h(\xi_1)| \left(1+\frac{|\xi_2|}{|\xi_1|} \right)^{-\gamma},
\end{eqnarray*}
where $\delta >6$, $\gamma \geq  3$, $h\in L^1(\R)$ and $C$ is a constant, and $\tilde{\psi}$ satisfies analogous conditions
with the roles of $\xi_1$ and $\xi_2$ exchanged. Then $\mathcal{SH}(\phi, \psi, \tilde{\psi}, c)$ provides an asymptotically optimal sparse approximation
of $f \in \cE^2(\R^2)$, i.e.,
\[
\norm{f - f_N}_2^2 =O(N^{-2} \cdot (\log N)^3)\quad \mbox{as } N \to \infty.
\]
\end{theorem}
Theorem~\ref{thm:ShearletOptimallySparseApproximation} indicates that shearlet systems provide a very good model for encoding the governing features of a scatterer.

\section{The Nonlinear Acoustic Scattering Problem} \label{sec:acoustic}
In this section we focus on the first of the two numerical approaches, which is to directly tackle the nonlinear problem. As a case study, we consider the
acoustic scattering problem. After briefly discussing the direct problem, we introduce the related inverse problem, which we approach
using a Tikhonov type functional with regularization by $\ell_p$ minimization, with $p$ close to or equal to $1$, applied to the shearlet coefficients.
For this setting, we will provide numerical examples and compare with other regularization approaches.
\subsection{The Direct Problem}
\label{subsec:directproblem}

A very common model for the behavior of an acoustic wave $u : \R^2 \to \C$ in an inhomogeneous medium is the Helmholtz equation \cite{ColK}. Given a \emph{wave number}  $k_0>0$  and a compactly supported \emph{contrast function} $f \in L^2(\R^2)$, the Helmholtz equation has the form
\begin{equation}
\Delta u + k_0^2 (1-f) u  = 0, \label{eq:Helmholtz}
\end{equation}
where the contrast function $f$ models the inhomogeneity of the medium due the scatterer. In a typical situation one models $f$ as a function which is smooth, apart from a model of the scatterer which is again assumed to be an essentially homogeneous medium, whose density is, however, significantly different from the surrounding medium. To model the boundary of the scatterer, a typical approach is to use a curve with a particular regularity, say $C^2$. Recalling the definition of a cartoon-like function in Definition \ref{def:CartoonLikeFunction}, we suggest to use $\mathcal{E}^2(\R^2)$ as a model for the boundary.
Furthermore, even if the boundary of the scatterer is only a piecewise $C^2$ curve, then
 the shearlets provide a very good model, since the sparse approximation results of shearlets as well as our analysis also hold in this more general situation.
As further ingredient for the acoustic scattering problem, we introduce
\emph{incident waves} $u^{\textrm{inc}}$, which are solutions to the \emph{homogeneous Helmholtz equation}, i.e., \eqref{eq:Helmholtz} with $f \equiv 0$. A large class of such solutions take the form
\beq \label{eq:solution_ud}
u^{\textrm{inc}}_d : \R^2 \to \C, \quad u^{\textrm{inc}}_d(x) = e^{i k_0 \ip{x}{d}}
\eeq
for some direction $d \in \mathbb{S}^1$. Then, for a given $f \in L^2(\R^2)$ and  a solution $u^{\textrm{inc}}$ to the homogeneous Helmholtz equation, every solution to \eqref{eq:Helmholtz} can be expressed as \hl{$u = u^{s} + u^{\textrm{inc}}$}, where $u^{s}$ denotes the \emph{scattered wave}. To obtain physically reasonable solutions we stipulate that the scattered wave obeys the \emph{Sommerfeld radiation condition}, see e.g. \cite{ColK},
\[
\frac{\partial u^s}{\partial |x|} = i k_0 u^s(x) + O(|x|^{-\frac{1}{2}}) \text{ for } |x| \to \infty.
\]
For a given $k_0>0$, and contrast function $f\in L^2(\R^2)$ with compact support and incident wave $u^{\textrm{inc}}$, the \emph{acoustic scattering problem} then is to find $u \in H^2_{loc}(\R^2)$ such that
\begin{eqnarray} \nonumber
 \Delta u + k_0^2 (1-f) u  &=& 0,\\ \nonumber
 u &=& u^s + u^{\textrm{inc}},\\ \nonumber
\frac{\partial u^s}{\partial |x|} &=& i k_0 u^s(x) + O(|x|^{-\frac{1}{2}}).
\end{eqnarray}
To obtain an equivalent formulation, we introduce the \emph{fundamental solution $G_{k_0}$ to the Helmholtz equation}, %
\begin{equation} \label{eq:defi_G}
G_{t}(x,y) = \frac{i}{4}H_0^{(1)}(t|x-y|), \quad t > 0, x,y \in \R^2,
\end{equation}
where $H_0^{(1)}$ is a \emph{Hankel function}, see e.g. \cite{abramowitz+stegun}. Let $B_R$ denote the open ball of radius $R > 0$ centered at $0$ and let $R$ be chosen such that $\suppp f \subset B_R$, then the \emph{volume potential} is
defined by
\[
 V(f)(x): = \int \limits_{B_R} G_{k_0}(x,y) f(y)\, dy, \quad x\in \R^2.
\]
Using this potential we can reformulate the acoustic scattering problem as the solution of the \emph{Lippmann-Schwinger integral equation} given by
\begin{equation}
 u^s(x) = - k_0^2 V(f (u^s + u^{\textrm{inc}})) \text{ in } B_R \label{eq:LippmannSchwinger}
\end{equation}
for $f\in L^2(\R^2)$ with $\suppp f \subset B_R$. Any solution $u^s \in H^2_{loc}(B_R)$ of \eqref{eq:LippmannSchwinger} indeed solves the acoustic scattering problem in $B_R$ and can, by the unique continuation principle \cite{JerK}, be uniquely extended to a global solution of the acoustic scattering problem, see e.g. \cite{ColK}.

Letting $L^2(B_R)$ denote the square-integrable functions defined on $B_R$, which are in particular compactly
supported, we now define the \emph{solution operator} of the acoustic scattering problem by
\[
\mathcal{S} : L^2(B_R) \times L^2(B_R) \to H^2_{loc}(B_R), \quad \mathcal{S}(f, u^{\textrm{inc}}) = u,
\]
%
and the Lippmann-Schwinger
equation \eqref{eq:LippmannSchwinger} allows to compute this operator for a given scatterer $f$ and
incident wave $u^{\textrm{inc}}$.

\subsection{The Inverse Problem}

In  the associated inverse problem,  we assume that we know
the incident wave $u^{\textrm{inc}}$ as well as measurements of the scattered wave $u^{s}$ and we aim to compute
information about the scatterer $f$. Following \cite{LKKLpRegularization2013}, we model these measurements as $u^s|_{\Gamma_{\mathrm{meas}}}$,
where $\Gamma_{\mathrm{meas}}$ is the trace of a closed locally Lipschitz continuous curve with $\Gamma_{\mathrm{meas}} \cap \overline{B_R} = \emptyset$.

In the case that we just have one incident
wave $u^{\textrm{inc}}$, then the map $(f, u^{\textrm{inc}})\mapsto u^s|_{\Gamma_{\mathrm{meas}}}$, is called
\emph{mono-static contrast-to-measurement operator} in \cite{LKKLpRegularization2013}. For multiple incident waves, and multi-static measurements, a closed set $\Gamma_{\textrm{inc}}$ is introduced, which is again the trace of a closed locally Lipschitz curve  enclosing $B_R$, such that $\Gamma_{\textrm{inc}} \cap \overline{B_R}= \emptyset$. The set $\Gamma_{\textrm{inc}}$ serves to construct \emph{single layer potentials}, which take the role of
the incident waves. For $\phi \in L^2(\Gamma_{\textrm{inc}})$, these single layer potentials are
\[
 SL_{\Gamma_{\textrm{inc}}}\phi := \int \limits_{\Gamma_{\textrm{inc}}} G_{k_0}(\cdot, y) \phi(y)\, dy \in L^2(B_R),
\]
see Figure \ref{fig:acousticmodel} for an illustration.
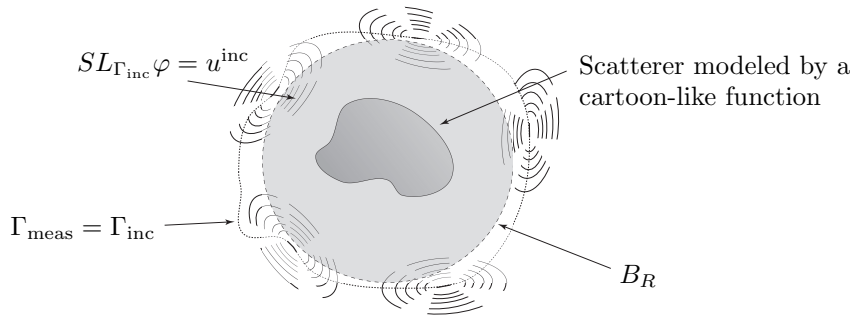
\begin{figure}[ht]
\centering
\includegraphics[height=4.5cm]{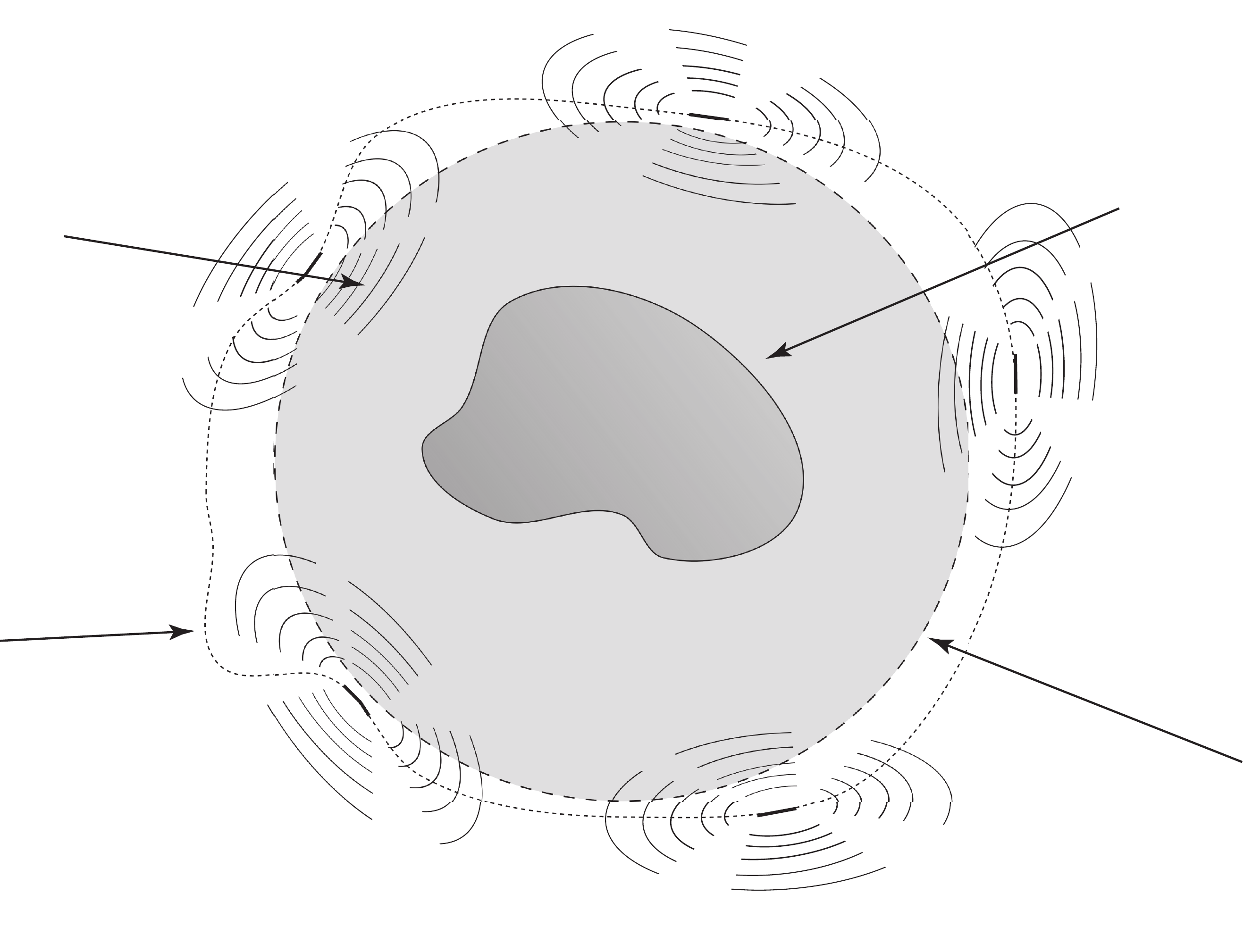}
\put(3,20){{\small $B_R$}}
\put(-12,99){{\small Scatterer modeled by a}}
\put(-12,87){{\small cartoon-like function}}
\put(-225,38){{\small $\Gamma_{\mathrm{meas}} = \Gamma_{\textrm{inc}}$}}
\put(-200,100){{\small $SL_{\Gamma_{\textrm{inc}}}\phi = u^{\textrm{inc}}$}}
\caption{Model for the acoustic inverse scattering problem in which $\Gamma_{\mathrm{meas}} = \Gamma_{\textrm{inc}}$.}
\label{fig:acousticmodel}
\end{figure}
Let $L^p_{\mathrm{Im}\geq 0}(B_R)$
denote the set of $L^p(B_R)$-functions with nonnegative imaginary part, and let $\mathrm{HS}(\cdot, \cdot)$ denote the space of Hilbert Schmidt operators \cite{LinearOperators1980}. Then the \emph{multi-static measurement operator} $\mathcal{N}$, which
assigns to each contrast function a Hilbert-Schmidt operator, maps a single layer potential to the associated
solution of the acoustic scattering problem. Formally, $\mathcal{N}$ is defined by
\[
\mathcal{N}: \ L^2_{\textrm{Im}\geq 0} (B_R) \to \mathrm{HS}(L^2(\Gamma_{\mathrm{inc}}), L^2(\Gamma_{\mathrm{meas}})),
\quad f \mapsto N_f,
\]
where
\[
N_f: L^2(\Gamma_{\mathrm{inc}}) \to L^2(\Gamma_{\mathrm{meas}}), \quad \phi \mapsto \mathcal{S}(f, SL_{\Gamma_{\textrm{inc}}}\phi)_{|\Gamma_{\mathrm{meas}}}.
\]
Note that indeed $N_f \in \mathrm{HS}(L^2(\Gamma_{\mathrm{inc}}), L^2(\Gamma_{\mathrm{meas}}))$, see \cite{LKKLpRegularization2013}, where it was also shown, even in a more general setting, that the operator $\mathcal{N}$ satisfies the following properties.
%
\begin{theorem}\cite{LKKLpRegularization2013}
The operator $\mathcal{N}$ is continuous, compact, and weakly sequentially closed from $L^2_{\mathrm{Im}\geq 0}(B_R)$ into $\mathrm{HS}(L^2(\Gamma_{\mathrm{inc}}), L^2(\Gamma_{\mathrm{meas}}))$.
\end{theorem}
Since in realistic applications the signals always contain noise, we consider
the \emph{inverse acoustic scattering problem} with noisy data, which
for a noise level $\epsilon > 0$, and noisy measurements $N^{\epsilon}_{\mathrm{meas}} \in \mathrm{HS}(L^2(\Gamma_{\mathrm{inc}}), L^2(\Gamma_{\mathrm{meas}}))$ satisfying
\begin{equation}
\| N^{\epsilon}_{\mathrm{meas}} - N_{f^\dagger} \|_{\mathrm{HS}(L^2(\Gamma_{\textrm{inc}}), L^2(\Gamma_{\mathrm{meas}}))} \leq \epsilon \label{eq:noisyMeasurementsObeyThis},
\end{equation}
is to recover the scatterer $f^\dagger$.
This inverse problem is ill-posed and requires careful regularization which is discussed in the next subsection.

\subsection{Regularization by Frames} \label{subsec:reg_frames}

A classical regularization approach to solve inverse problems is to minimize an appropriate Tikhonov functional. If $X$ is a Hilbert space and $Y$ a Banach spaces and $F:X \to Y$ a possibly nonlinear operator, $\mathcal{R}:X\to \R$ a convex functional, $\alpha>0$ a regularization parameter and $y^\epsilon$ a noisy datum, then the Tikhonov functionals under consideration are of the form:
$$T(x) = \frac{1}{2}\|F(x) - y^\epsilon\|_Y^2 + \alpha \mathcal{R}(x), \text{ for }x\in X.$$
This methodology was introduced by Tikhonov, who used this method to solve linear inverse problems and employed a Sobolev norm as penalty term, \cite{Tikhonov1963}. Later, sparsity promoting regularizations where used, where for instance for $0\leq p\leq 2$, $\mathcal{R}(x)=\sum_i|\left \langle x, \phi_i \right \rangle|^p$, with an orthonormal basis $(\phi_i)$ of $X$, \cite{DauDD,GraHS2008}.
Further extensions of this concept to nonlinear operators $F$ and penalty terms $\mathcal{R} = \sum_i|\left \langle x, \phi_i \right \rangle|^p$ with some frame $(\phi_i)_i$ have been described in various flavors in e.g. \cite{JinM2012,RamT2005,TecB2010}.

In \cite{LKKLpRegularization2013} it was suggested to minimize
\begin{equation}
\tilde{\mathfrak{T}}_\alpha^{\epsilon}(f) : = \frac{1}{2}\left\|N_f - N^{\epsilon}_{\mathrm{meas}}  \right\|^2_{\mathrm{HS}(L^2(\Gamma_{\textrm{inc}}), L^2(\Gamma_{\mathrm{meas}}))} + \frac{\alpha}{p} \|f\|^p_{L^p(B_R)}, \quad  f \in L^p(B_R).  \label{eq:LechleiterFunctional}
\end{equation}
for fixed $p>1$ and $\alpha> 0$.
For $p$ close to $1$, the regularization term in this functional promotes sparsity in the representation of the scatterer. Regularization techniques of this sort have been studied comprehensively, see e.g. \cite{SchGG2009}.

Here we  suggest a different regularization, which exploits that
the scatterer $f$ is modeled as a cartoon-like function $\cE^2(\R^2)$ and that
shearlet systems are used to obtain sparse approximations of the scatterer $f$. Let $\Phi :=
\mathcal{SH}(\phi, \psi, \tilde{\psi}, c)$  be a shearlet frame satisfying the hypotheses of Theorem \ref{thm:ShearletOptimallySparseApproximation}, and let
$T_\Phi$ denote the analysis operator of the shearlet frame $\Phi$, then the proof of Theorem \ref{thm:ShearletOptimallySparseApproximation} yields the decay behavior of the associated shearlet coefficients $T_\Phi(f)$.
It has been shown in \cite{KLcmptShearSparse2011} that  $T_\Phi(f)$ is in $\ell^p$
for every $p>\frac23$.
%

We propose to regularize the acoustic inverse scattering problem by adapting the data
fidelity term appropriately and by imposing a constraint on the $\ell^p$-norm of the coefficient sequence $T_\Phi(f)$. More precisely, for fixed
$1 \le p \le 2$ and $\alpha> 0$, we consider the Tikhonov functional 
\begin{equation}\label{eq:shearMin}
\mathfrak{T}_\alpha^{\epsilon}(f) : = \frac{1}{2}\left\|\mathcal{N}( f ) - N^{\epsilon}_{\mathrm{meas}}
\right\|^2_{\mathrm{HS}(L^2(\Gamma_{\textrm{inc}}), L^2(\Gamma_{\mathrm{meas}}))} + \frac{\alpha}{p} \|T_\Phi(f)\|^p_{\ell^p}, \quad  f \in L^{2}(B_R).
\end{equation}
%
Note that the case of $p=1$ is not excluded in our analysis in contrast to the situation in \cite{LKKLpRegularization2013}.

Having introduced the Tikhonov regularization, we now describe the convergence of the minimization process. For this we are particularly interested in
convergence to a \emph{norm minimizing solution} $f^* \in L^2(B_R)$, i.e.,
\[
 \mathcal{N}(f^*) = N_{f}  \quad \text{and} \quad \|T_\Phi(f^*)\|_p \leq \|T_\Phi(f)\|_p \text{ for all } f \mbox{ such that }
\mathcal{N}(f) = N_{f^\dagger}.
\]
Such convergence properties have been extensively studied. Let us give the following result, which combines results from \cite{JinM2012} directly for our setting.
\begin{theorem}
Let $\Phi$ be a shearlet frame, $\mathcal{N}$ the multistatic measurement operator, $\epsilon >0$, and $1\leq p\leq 2$.
Then the following assertions hold.
\begin{compactenum}
\item For every $\alpha>0$ there exists a minimizer of $\mathfrak{T}_\alpha^{\epsilon}$.
\item If $\alpha(\delta) \to 0$ and $\frac{\delta^2}{\alpha(\delta)}\to 0$ for $\delta \to 0$, then every sequence of minimizers $f^\delta_{\alpha(\delta)}$ has a subsequence, that converges in $L^2(\R^2)$ to a norm minimizing solution.
\end{compactenum}
\end{theorem}
Having introduced the necessary theoretical tools, we can now give numerical examples to compare the regularization by shearlet frames to that of \cite{LKKLpRegularization2013}.
%
%
\subsection{Numerical Methods for the Acoustic Inverse Scattering Problem}
\label{sec:numEx}

In this section, we will analyze numerical approaches to solve the acoustic scattering problem of Section \ref{sec:acoustic}. After discussing an algorithmic realization of our approach \eqref{eq:shearMin}, we briefly present the other numerical methods that we compare with, followed by a detailed description of the numerical experiments. It will turn out that our new method is advantageous to the other methods in the situation that the scatterer is a body consisting of a more or less homogeneous medium, whose density is significantly different from the surrounding medium.

\subsubsection{The New Algorithmic Approach} \label{subsec:algorithmicapproach}

As suggested in Subsection \ref{subsec:reg_frames}, we aim to solve the minimization problem, compare also \eqref{eq:shearMin},
\[
\min_{f \in L^2(B_R)} \left( \frac{1}{2}\left\|\mathcal{N}(f) - N^{\epsilon}_{\mathrm{meas}}
\right\|^2_{\mathrm{HS}(L^2(\Gamma_{\textrm{inc}}), L^2(\Gamma_{\mathrm{meas}}))} + \frac{\alpha}{p} \|T_{\Phi}(f)\|_{\ell^p}^p\right),
\]
where $\Phi$ is a shearlet frame for $L^2(\R^2)$, and $\tilde{\Phi}$ denotes the associated canonical dual frame. Employing the sign function, the mapping
\[
J_q: B_R \to \C, \quad x \mapsto [J_p(q)](x) := |q(x)|^{p-1}\mathrm{sign}(q(x)),
\]
and the operator $\mathcal{S}_{\alpha\mu, p}: = (I+\alpha\mu J_p)^{-1}$, it has been shown in \cite{LKKLpRegularization2013} that the solution via the standard Tikhonov functional \eqref{eq:LechleiterFunctional} can be obtained as the limit of the \emph{Landweber iteration}
\begin{equation} \label{eq:l1reg_solver}
 f_{n+1} = \mathcal{S}_{\alpha\mu_n, p}\left[ f_n - \mu_n [ \mathcal{N}'(f_n)]^*(\mathcal{N}(f_n)-N^{\epsilon}_{\mathrm{meas}})\right]
\quad \mbox{for } (\mu_n)_n \subset \R^+.
\end{equation}

By \cite{NonlinearTikhonovSparsityConstrains}, the solution to \eqref{eq:shearMin} with a frame based
 regularization term can be computed  as a limit of the iteration
\begin{equation}
 f_{n+1} = T_{\tilde{\Phi}^*}\mathcal{S}_{\alpha\mu_n, p}\left[T_{\Phi}( f_n - \mu_n [ \mathcal{N}'(f_n)]^*(\mathcal{N}(f_n)-N^{\epsilon}_{\mathrm{meas}}))\right]
\quad \mbox{for } (\mu_n)_n \subset \R^+,\label{eq:theIteration}
\end{equation}
see  \cite{LKKLpRegularization2013} for an explicit construction of $\mathcal{N}'$ and $[\mathcal{N}']^*$.

Since the $\ell_1$-norm promotes sparsity, in our experiments we will choose $p=1$. In this case,
$\mathcal{S}_{\alpha\mu, 1}$ is the soft-thresholding operator, which for a scalar $\omega$, is defined as
\[
\mathcal{S}_{\alpha, 1}(\omega) = \left\{ \begin{array}{l l l}
                                     \omega-\alpha, \quad &\text{ if } \ \omega &\geq \alpha,\\
                                     0, \quad & \text{ if } |\omega| &< \alpha,\\
                                    \omega + \alpha, \quad &\text{ if } \ -\omega &\leq -\alpha,\\
                                     \end{array}\right.
\]
with element-wise application for sequences.


The general setup of the numerical experiments, whose results will be described in Subsection \ref{subsec:exp_results}, follows that for similar experiments presented in \cite{LKKLpRegularization2013}.
We chose the stepsize $\mu_n$ according to the Barzilai-Borwein rule \cite{BarB1988}, and stop the iteration when
\begin{equation}
\| \mathcal{N}(f_n) - N^{\epsilon}_{\mathrm{meas}}\|_{\mathrm{HS}(L^2(\Gamma_i), L^2(\Gamma_m))} \leq \tau \epsilon, \label{eq:eq:theDiscrPrinciple}
\end{equation}
with  $\tau = 1.6$ and $\epsilon$ being a fixed parameter chosen according to the noise level. For an analysis of this stopping rule see, e.g. \cite{kaltenbacher2008iterative}.

Furthermore, we choose the regularization parameter $\alpha$ so that it determines the optimal value for a low noise level and then decrease it proportionally to the noise level $\epsilon$. Although other parameter choice rules have been proposed, as for instance the Morozov discrepancy principle \cite{AnzR2010}, we prefer this simpler a-priori choice, since this has also been used in the comparison results \cite{LKKLpRegularization2013}.

In each step of \eqref{eq:theIteration}, one shearlet decomposition and reconstruction step needs to be performed for which \emph{Shearlab} \cite{shearlab}
is used. In all experiments, a discretization of the domain with a $512 \times 512$  grid is used and the shearlet system of Shearlab using $5$ scales. \hl{We subsample the shearlet system according to Definition \ref{def:shearletSystem} using a mask.}

We select as domain $[-1,1]^2$, and let the scatterer be supported in $B_R$ with $R = 0.75$. We then pick $T$ transmitter-receiver pairs equidistributed
on the circle of radius $0.9$, \hl{where $T = 18$ in the first experiment and $T = 32$ in the second experiment}. Thus \hl{$2T$} Lippmann-Schwinger equations need to be solved in every step, \hl{$T$} for the evaluation of $\mathcal{N}(f_n)$
for the different single layer potentials, and \hl{$T$} for the evaluation of $[\mathcal{N}'(f_n)]^*$.

We then solve these equations with a simple, and admittedly slower, method than \cite{LKKLpRegularization2013}, by discretizing, and then solve the resulting linear system using a GMRES
iteration without preconditioning. The results in Subsection \ref{subsec:exp_results} show that even with this simple approach the advantage of the shearlet
regularization over other regularization methods can be observed. The increased runtime per step does not affect the overall runtime significantly, since we require less
iterations. However, using the method of \cite{LKKLpRegularization2013} to solve the Lippmann-Schwinger equations should provide a significant further speed-up in our algorithm, when aiming for higher numerical efficiency and not only for the accuracy of the reconstruction.

As scatterers $f$, we consider prototypes of cartoon-like functions, as in Figures~\ref{fig:Experiment1} and \ref{fig:Experiment2}.

\subsubsection{Comparison Results}
\label{subsec:exp_results}
We compare our approach with two other approaches, first the method introduced in \cite{LKKLpRegularization2013},
which is based on the assumption that the scatterer is itself sparse and hence an $L^1$ regularization is used, solving
\[
\min_{f \in L^2(B_R)} \left( \frac{1}{2}\left\|N_f - N^{\epsilon}_{\mathrm{meas}}  \right\|^2_{\mathrm{HS}(L^2(\Gamma_{\textrm{inc}}), L^2(\Gamma_{\mathrm{meas}}))} + \alpha \|f\|_{L^1(B_R)}\right)
\]
via the Landweber iteration \eqref{eq:l1reg_solver}, and second
\[
\min_{f \in L^2(B_R)} \left\|N_f - N^{\epsilon}_{\mathrm{meas}}  \right\|_{\mathrm{HS}(L^2(\Gamma_{\textrm{inc}}), L^2(\Gamma_{\mathrm{meas}}))},
\]
which does not contain a regularization term, hence does not exploit sparsity in any way. We stop the iteration when \eqref{eq:eq:theDiscrPrinciple} is achieved. This method with the regularization term based on the $L^1(B_R)$ is not covered by the theoretical results in \cite{LKKLpRegularization2013} which only yield convergence results for $L^p$ norms with $p>1$. Since this approach is used for the numerical results in \cite{LKKLpRegularization2013}, we also compare with this method rather than with an $L^p$ penalty for $p>1$.


In the first set of experiments we choose a wave number of $k_0 = 20$ and compute reconstructions with our approach, see Subsection
\ref{subsec:algorithmicapproach}.
The different noise levels that we impose are described in Table~\ref{tab:1} and Figure~\ref{fig:Experiment1}. In Table \ref{tab:1}, for each of the three regularization methods, we provide the relative error measured in the discrete $L^2$-norm as well as the number of iterations until \eqref{eq:eq:theDiscrPrinciple} is achieved for different noise levels.

\begin{figure}[htb]
     \centering

        \centering
       \quad \ \includegraphics[width=0.345\textwidth]{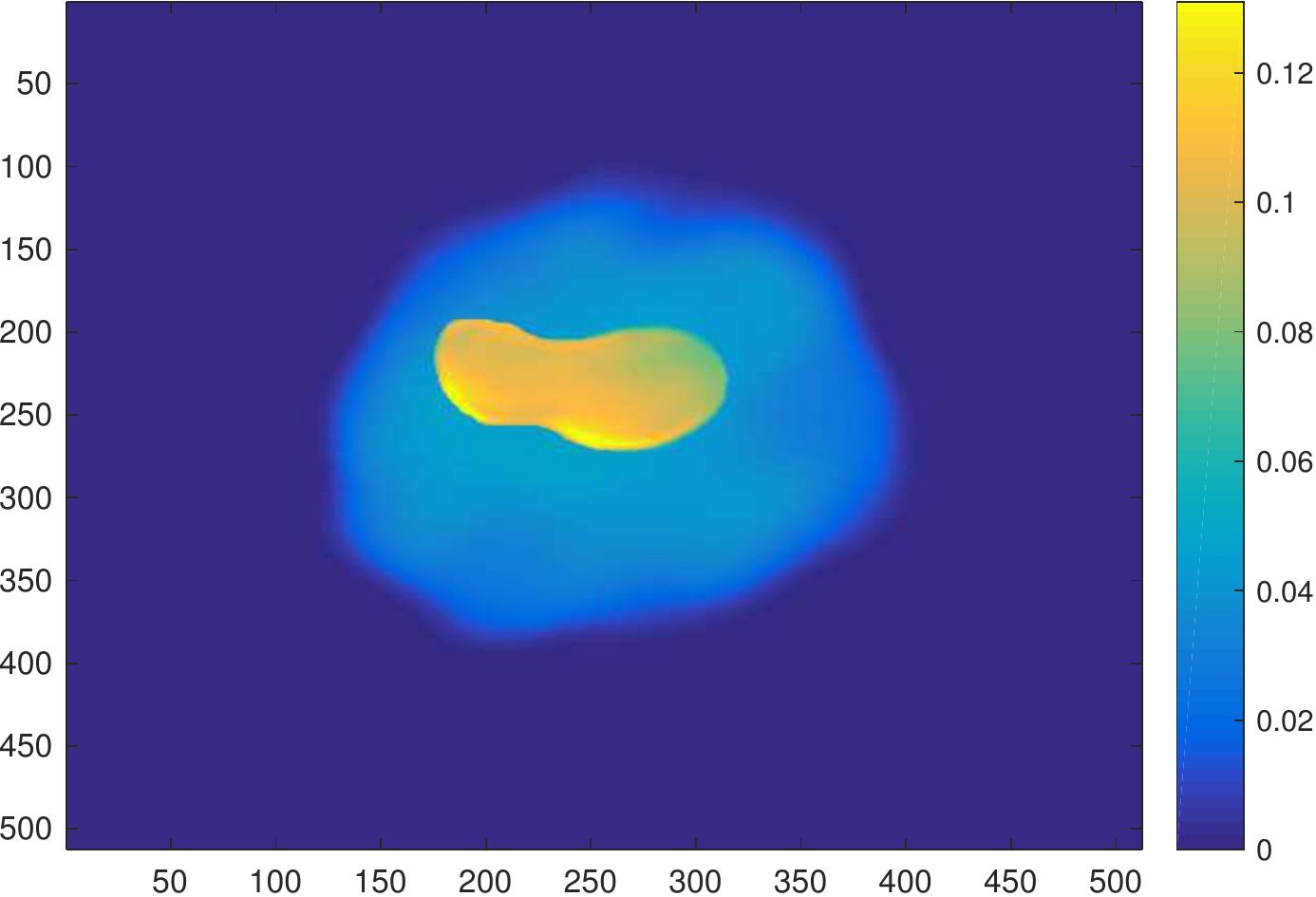}
    \\
        \centering
        \includegraphics[width=0.30\textwidth]{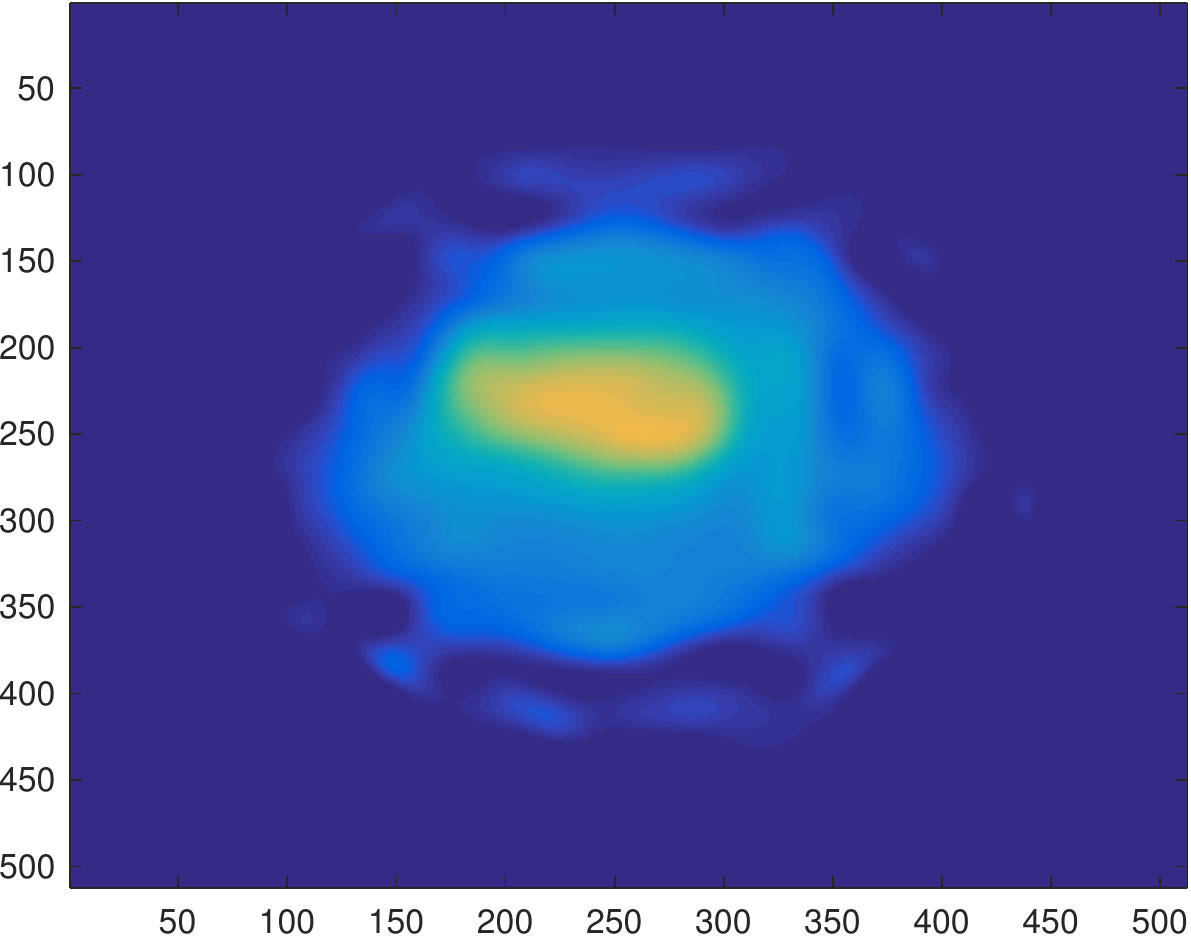}
        \centering
        \includegraphics[width=0.30\textwidth]{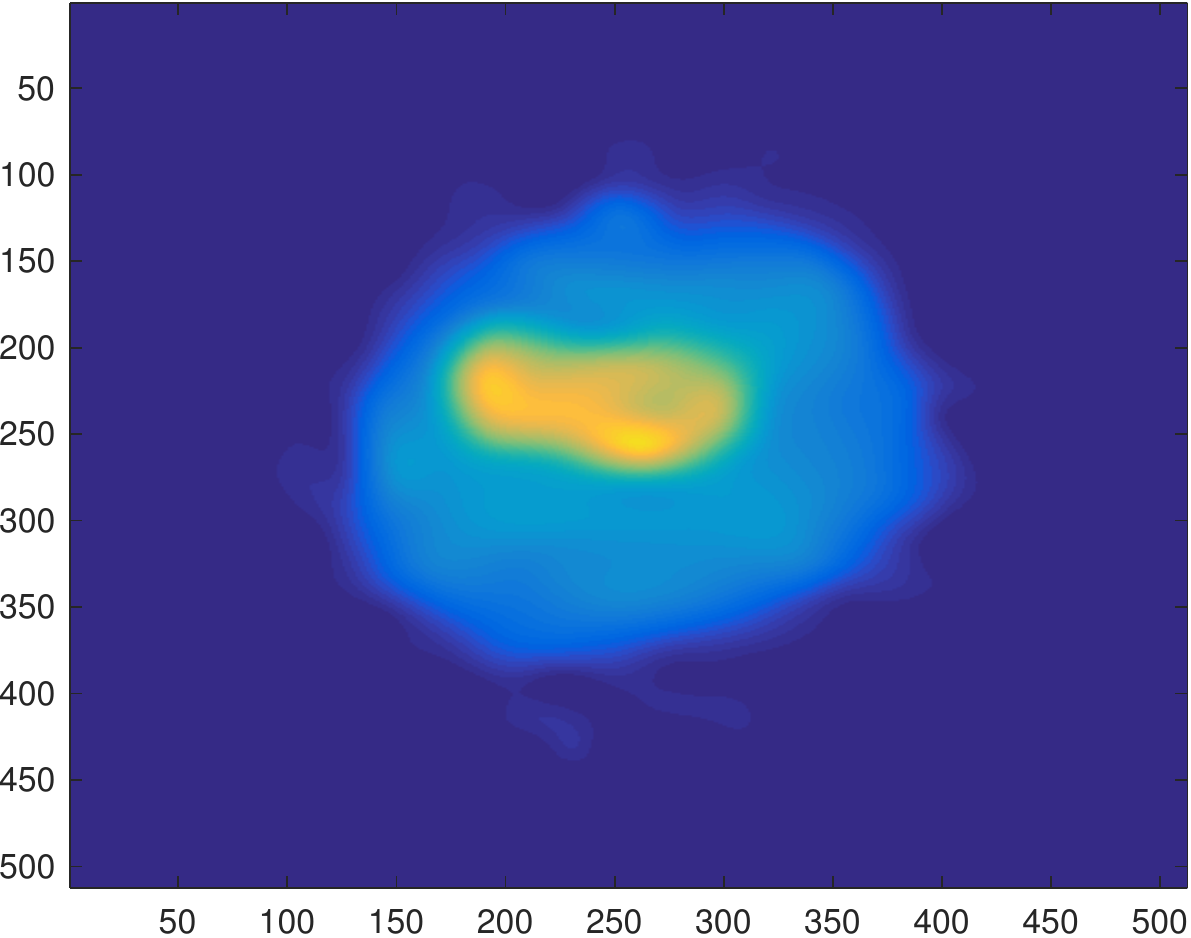}
        \centering
      \includegraphics[width=0.30\textwidth]{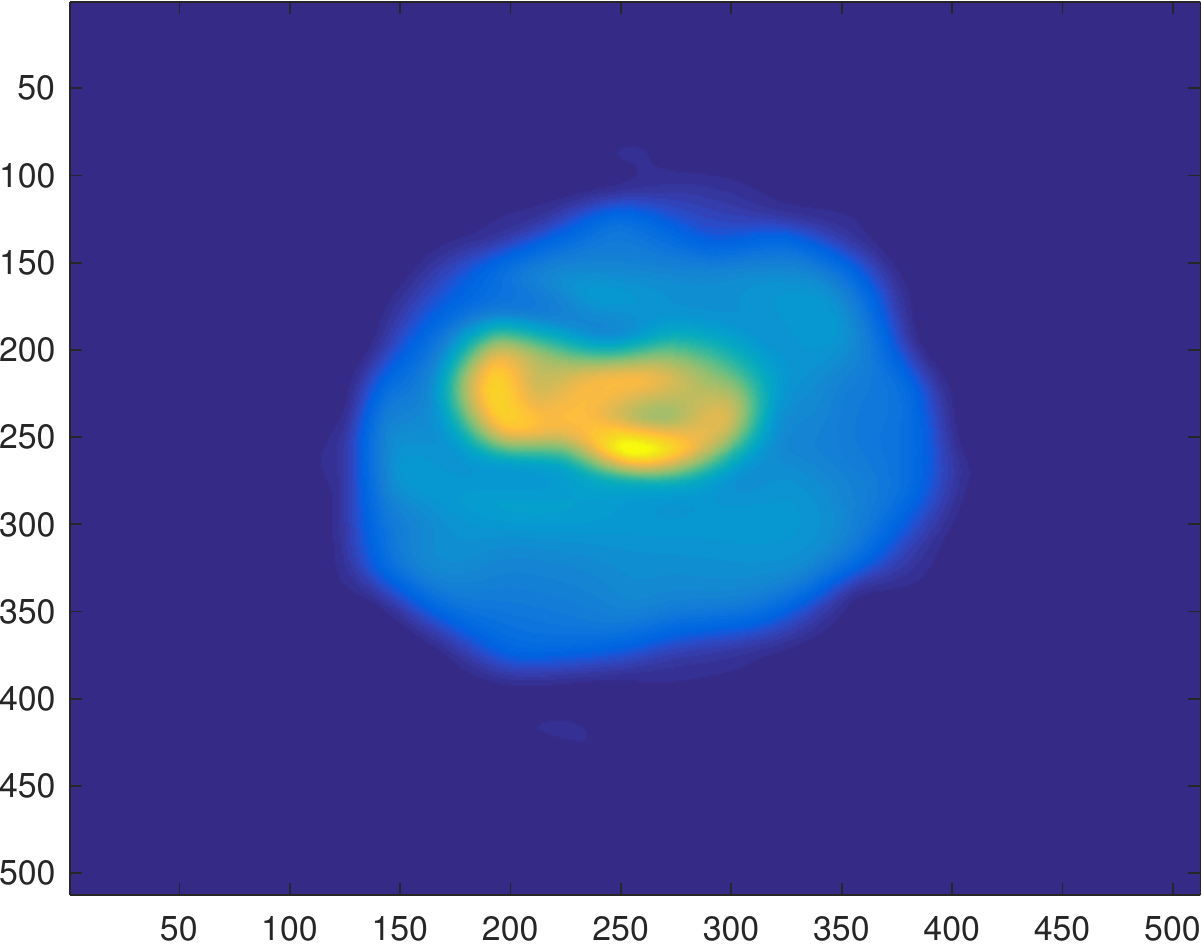}
        \\
        \centering
        \includegraphics[width=0.30\textwidth]{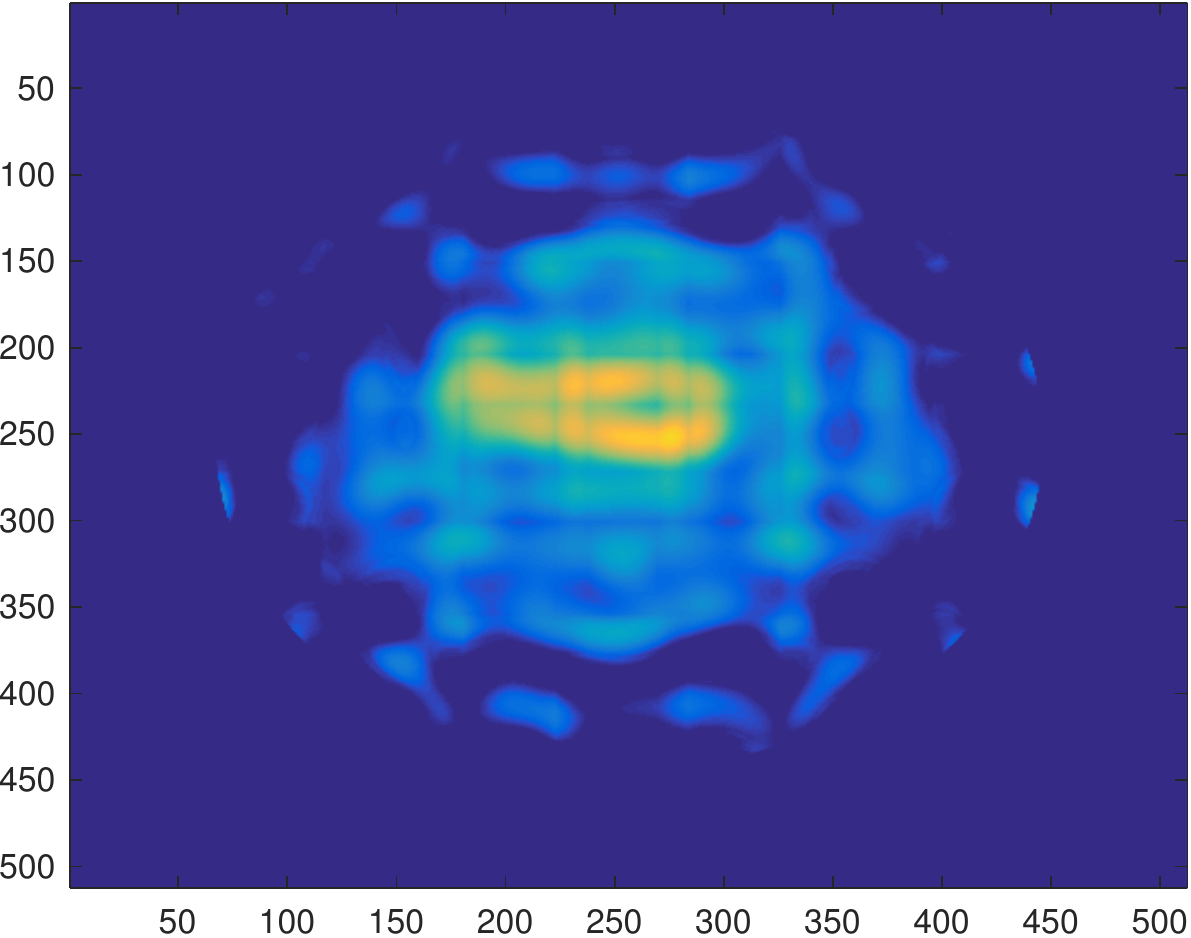}
        \centering
        \includegraphics[width=0.30\textwidth]{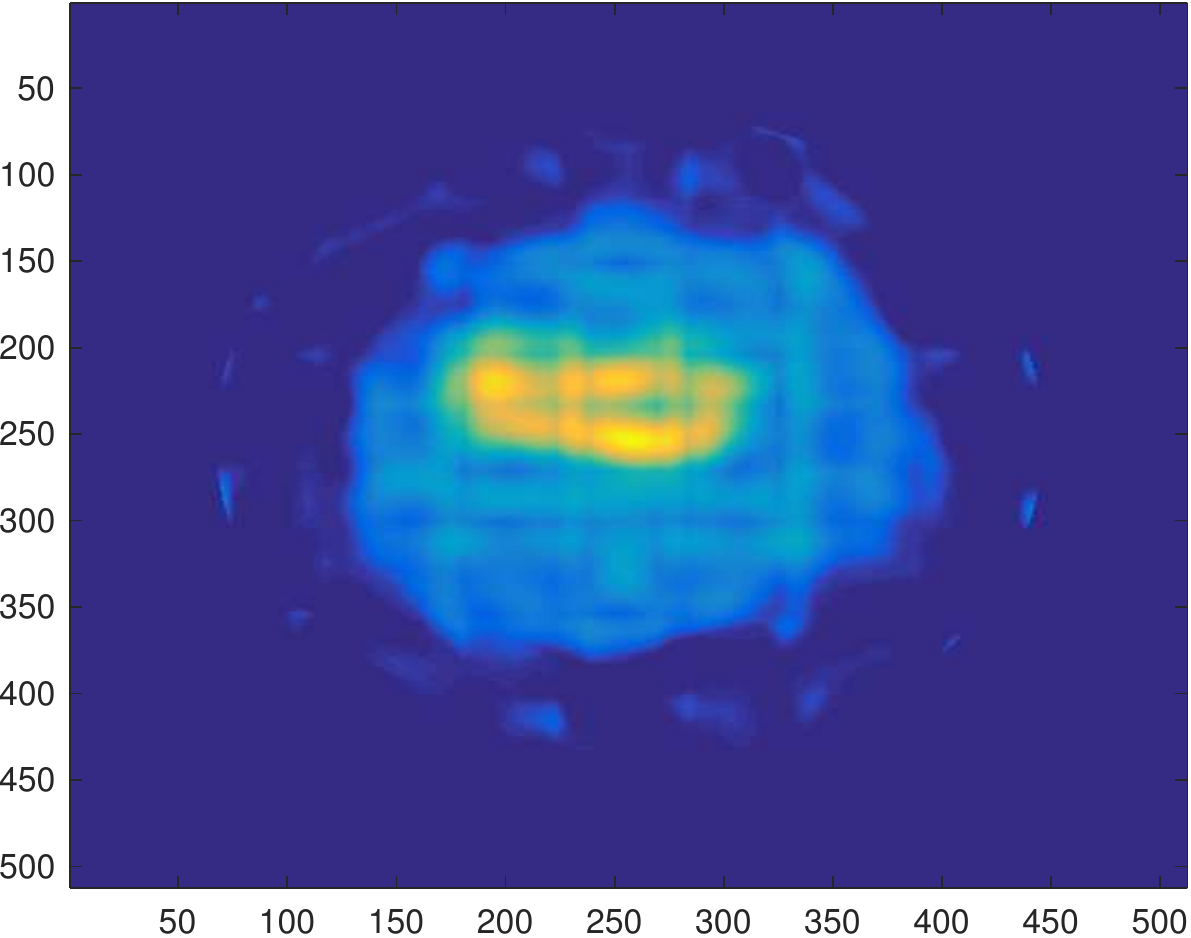}
        \centering
        \includegraphics[width = 0.30\textwidth]{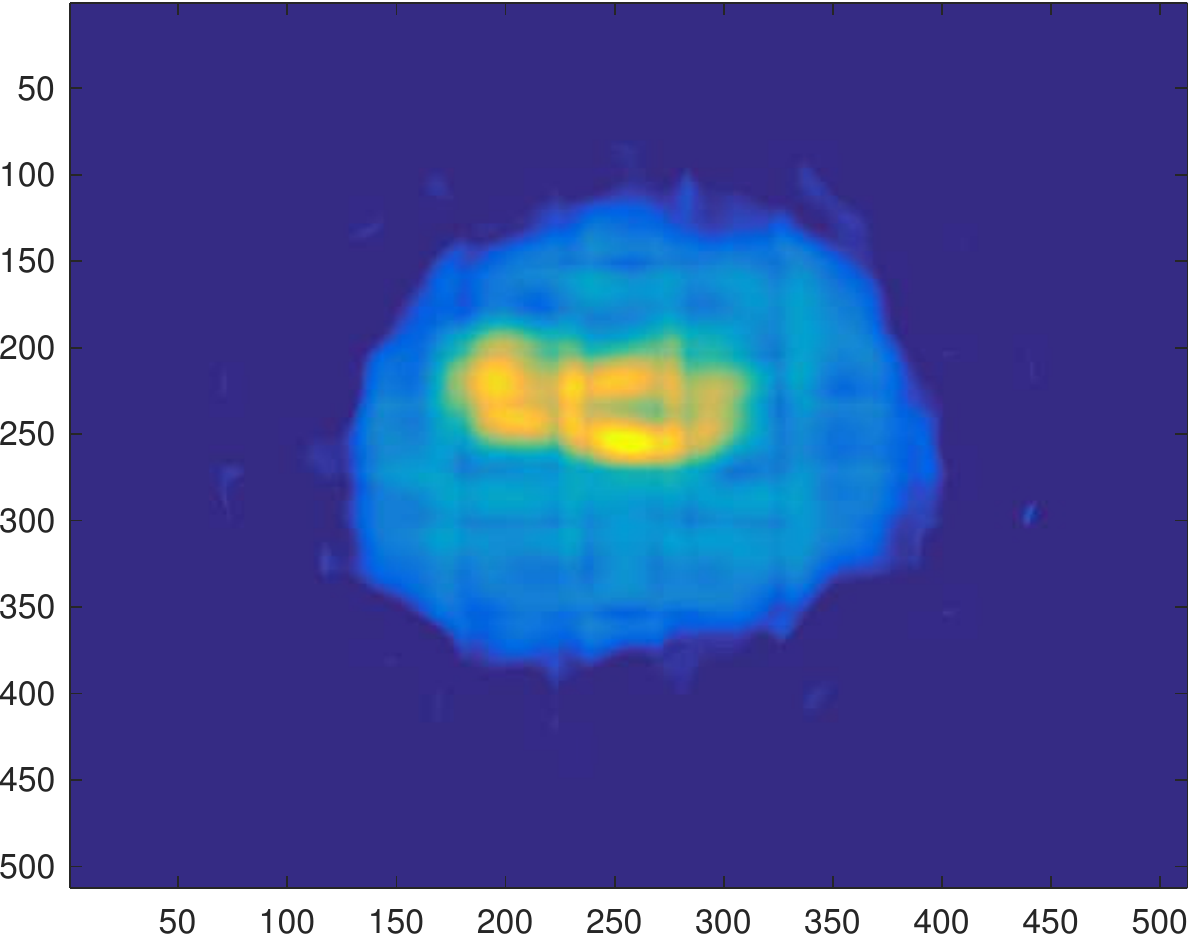}
    \caption{\textbf{Top:} \hl{Cartoon-like} scatterer. \textbf{Second Row:} Reconstructed scatterers using the shearlet regularization, \hl{relative noise levels from left to right: $\epsilon = 0.08, 0.02, 0.005$}. \textbf{Third Row:} Reconstructed scatterers using the $L^1$ regularization, \hl{relative noise levels from left to right: $\epsilon = 0.08, 0.02, 0.005$}. \hl{The colormaps in all images coincide with that of the top image.}}
    \label{fig:Experiment1}
\end{figure}

\begin{table}[!h]

\scriptsize  \centering
    \begin{tabular}{ | c | l |r | r| r | }
    \hline
    $\mathbf{k_0 = 20}$ & \ Regularization method \ & \ rel.Noise lvl \ & \ Relative error \ & \  \#iterations \ \\ \hline
    1. \quad \ & \ $L^1$ Tikhonov & 0.08 \ & 0.2153 & 7    \   \\
    2. \quad \ & \ Shearlets & 0.08 \ & \textbf{0.1274} &  7 \       \\
    3. \quad \ & \ No Penalty &  0.08 \ & 0.2188 &  \  9  \ \\ \hline
    4. \quad \ & \ $L^1$ Tikhonov  & 0.04 \ &  0.1435 & 13 \        \\
    5. \quad \ & \ Shearlets & 0.04 \ & \textbf{0.0920} &  \  12  \ \\
    6. \quad \ & \ No Penalty & 0.04 \ &  0.1567 & \   12   \ \\ \hline
    7. \quad \ & \ $L^1$ Tikhonov & 0.02 \ & 0.1162 &  \  19  \ \\
    8. \quad \ & \ Shearlets &  0.02  \ &  \textbf{0.0723} & \   19   \ \\
    9. \quad \ & \ No Penalty &  0.02  \ &  0.1161 & \   19   \ \\ \hline
    10. \quad \ & \ $L^1$ Tikhonov &  0.005  \ & 0.0848 &  \  96  \ \\
    11. \quad \ & \ Shearlets & 0.005 \ &  \textbf{0.0665} & \   48   \ \\
    12. \quad \ & \ No Penalty & 0.005 \ &  0.0930 & \   83   \ \\
    \hline
    \end{tabular}
\caption{Numerical results of the three regularization methods for different noise levels with the scatterer chosen as in Figure \ref{fig:Experiment1}.}
\label{tab:1}
\end{table}


The shearlet scheme shows the best performance both visually and with respect to the relative error.  The inferior performance of the $L^1$ regularization from \cite{LKKLpRegularization2013}  is due to the fact that the scatterer is not sparse itself in the sense of having a relatively small support. Certainly, if no penalty term is used, then the solution is not at all adapted to the  specific structure and expectedly, the performance is worse.

\begin{figure}[!h]
     \centering  \ \hspace{0.25cm}
        \includegraphics[width=0.375\textwidth]{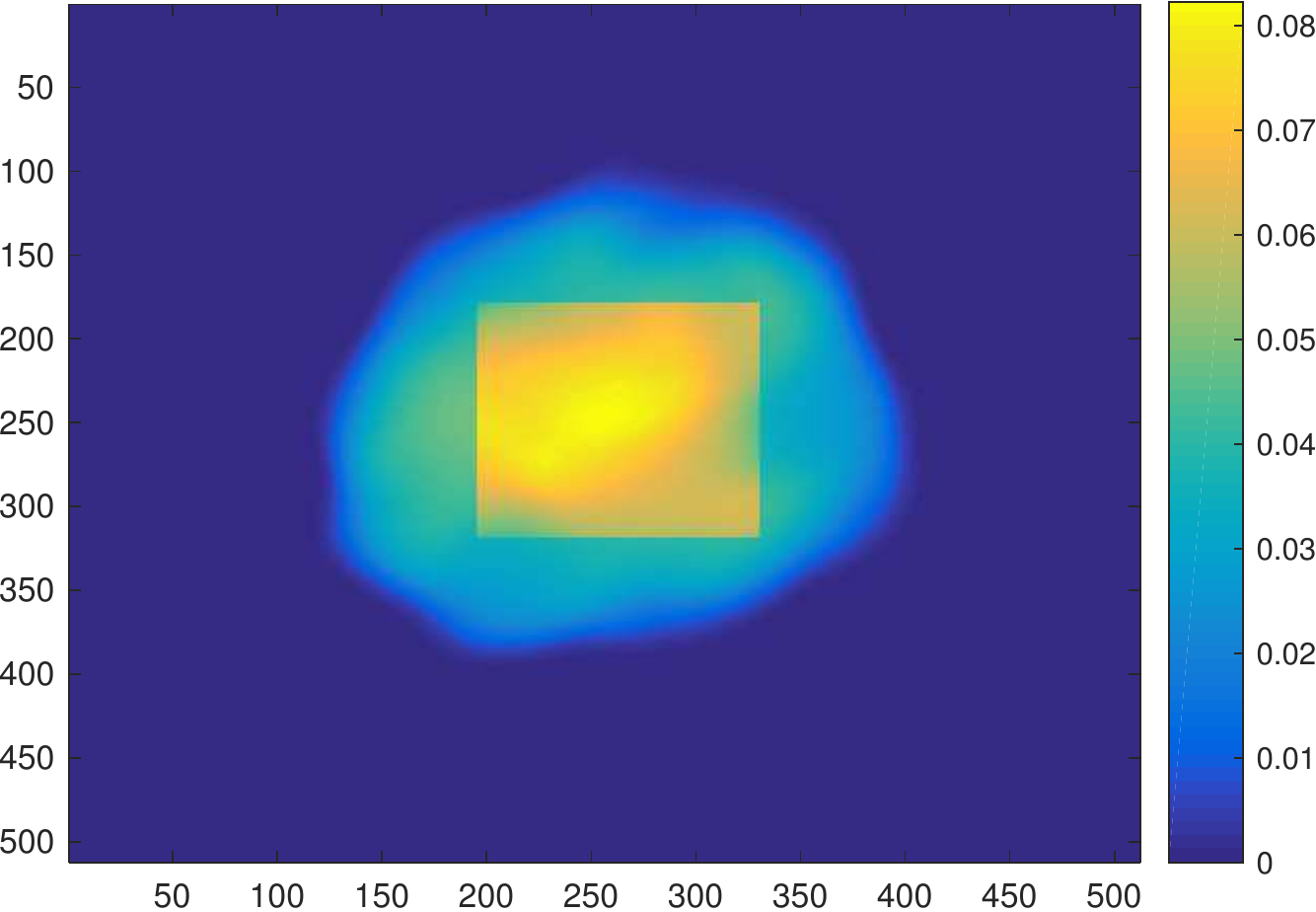}

        \centering
        \includegraphics[width=0.325\textwidth]{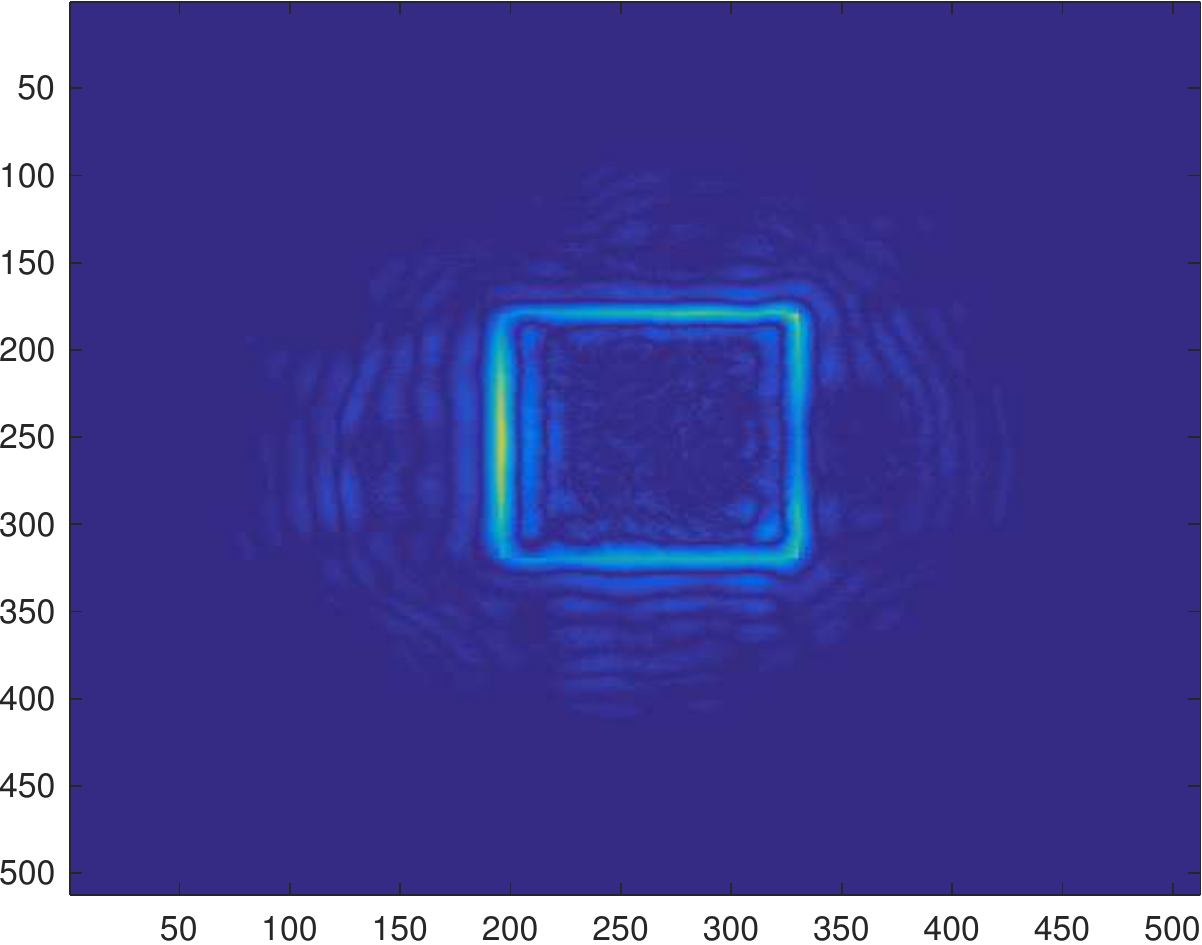}~\includegraphics[width=0.325\textwidth]{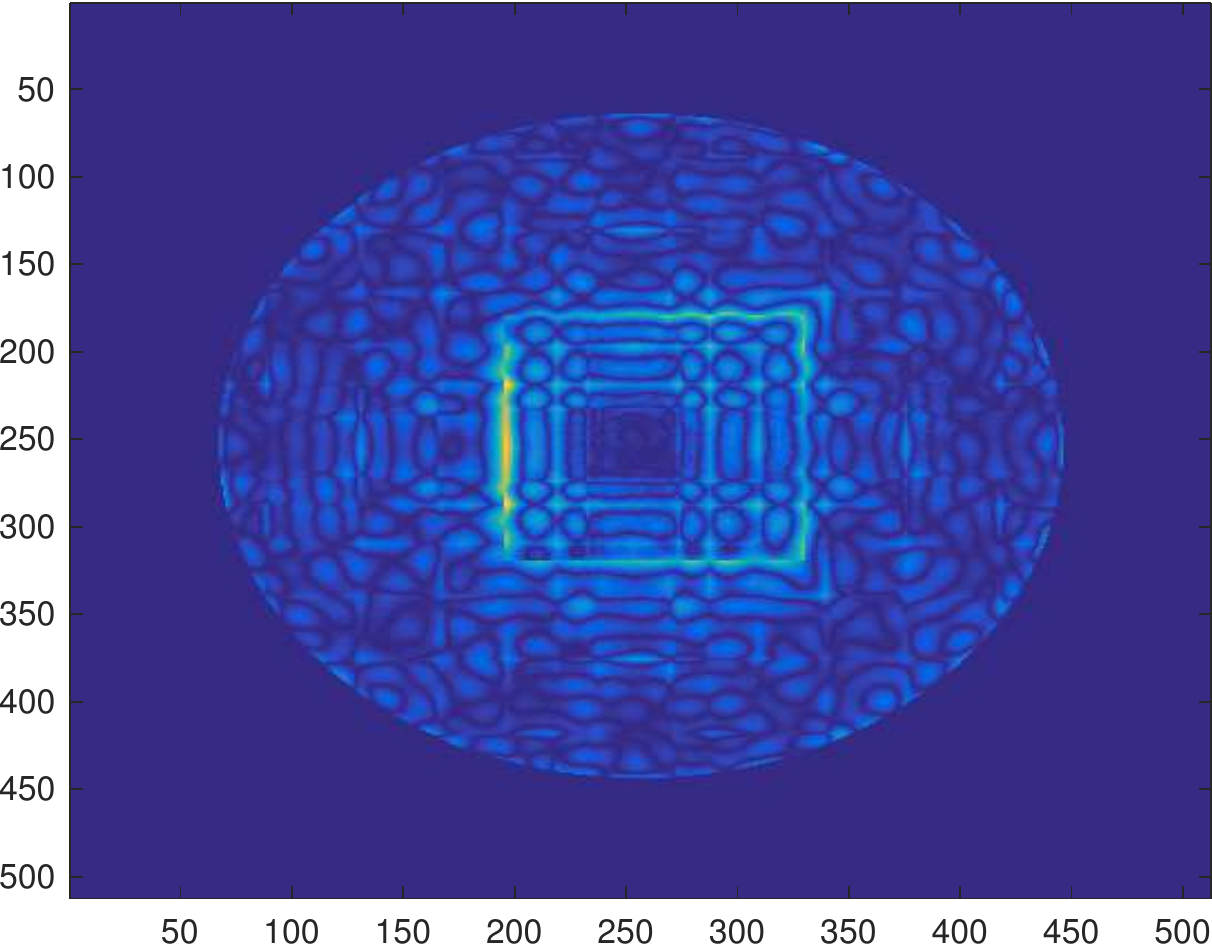}~\includegraphics[width=0.378\textwidth]{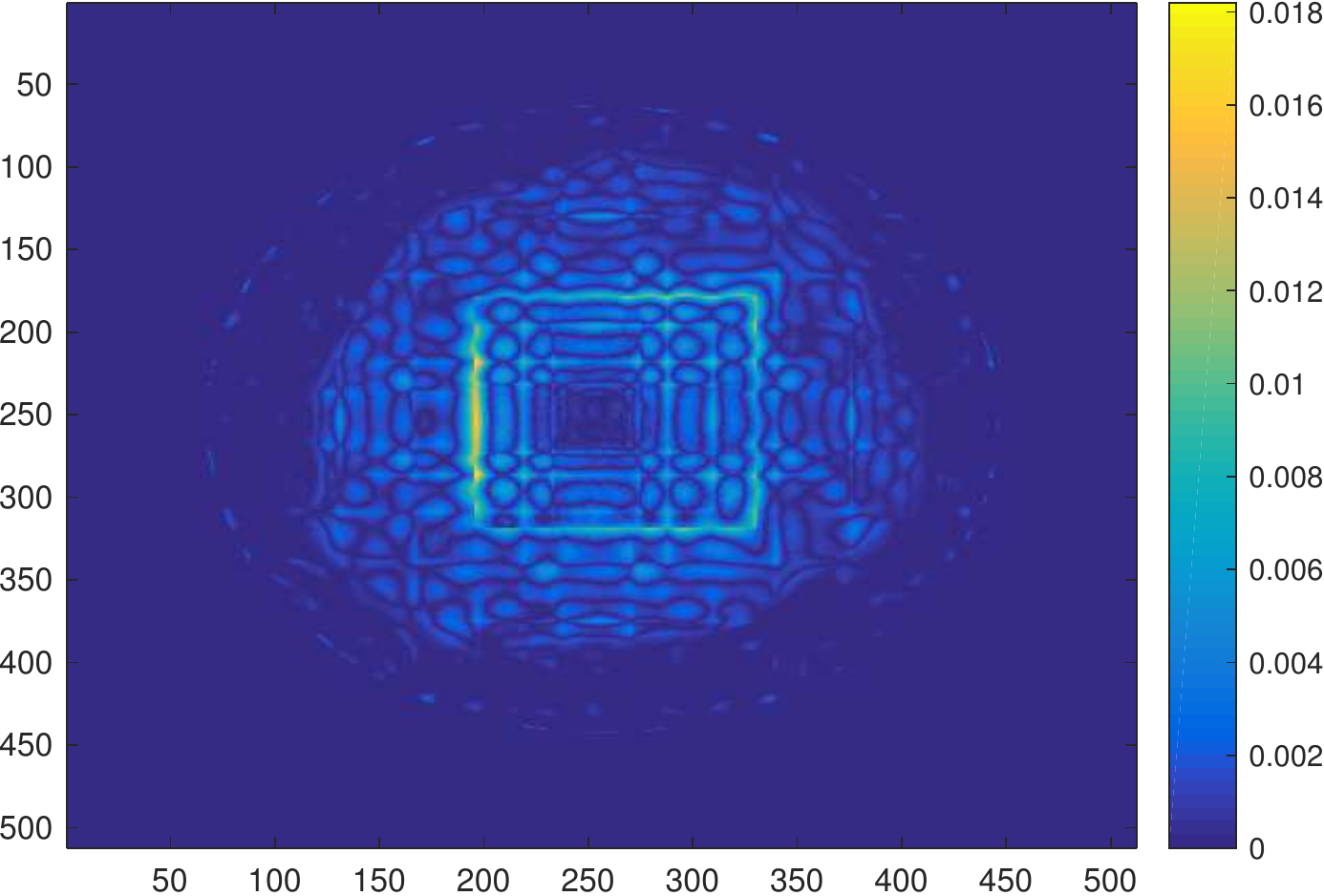}
    \caption{\hl{\textbf{Top:} Cartoon-like scatterer with piecewise smooth jump curve, \textbf{Second Row:(from left to right)} 
    Error of the reconstruction of the cartoon-like scatterer from the top row using the shearlet regularization, using no penalty term, and using $L^1$ Tikhonov Regularization. The relative noise level is 0.001. }}
        \label{fig:Experiment2}
\end{figure}

Top: Cartoon-like scatterer with piecewise smooth jump curve,
Second Row: (from left to right) Error of the reconstruction of the cartoon-like scatterer using the shearlet regularization, using no penalty term, and using L1 Tikhonov Regularization. The relative noise level is 0.001. 

\begin{table}[h!]

\scriptsize  \centering
    \begin{tabular}{ | c | l |r | r| r | }
    \hline
    $\mathbf{k_0 = 25}$ & \ Regularization method \ & \ rel. Noise level \ & \ Relative error \ & \  \#iterations \ \\ \hline
    1. \quad \ & \ $L^1$ Tikhonov & 0.02 \ & 0.0714 & 14    \   \\
    2. \quad \ & \ Shearlets & 0.02 \ & \textbf{0.0426} & 14  \       \\
    3. \quad \ & \ No Penalty & 0.02 \ & 0.0717 &  \  16  \ \\ \hline

    4. \quad \ & \ $L^1$ Tikhonov  & 0.01 \ &  0.0587 & 23 \        \\
    5. \quad \ & \ Shearlets & 0.01 \ & \textbf{0.0383} &  \  19  \ \\
    6. \quad \ & \ No Penalty & 0.01 \ &  0.0620 & \   23   \ \\ \hline

    7. \quad \ & \ $L^1$ Tikhonov & 0.005 \ & 0.0561 &  \  32  \ \\
    8. \quad \ & \ Shearlets & 0.005 \ &  \textbf{0.0346} & \   33   \ \\
    9. \quad \ & \ No Penalty & 0.005 \ &  0.0577 & \   42   \ \\ \hline

    10. \quad \ & \ $L^1$ Tikhonov & 0.001 \ & 0.0420 &  \  271  \ \\
    11. \quad \ & \ Shearlets & 0.001 \ &  \textbf{0.0334} & \   72  \ \\
    12. \quad \ & \ No Penalty & 0.001 \ & 0.0429 & \   277   \ \\\hline
    \end{tabular}
\caption{Numerical results of the three regularization methods for different noise levels with the scatterer chosen as in \ref{fig:Experiment2}.}
\label{tab:2}
\end{table}

We also conducted a second set of experiments with a different wave number, i.e., $k_0=25$ and display the results in Table \ref{tab:2}.
Furthermore, the reconstruction error is depicted in Figure \ref{fig:Experiment2}, where we observe that the shearlet regularization produces satisfying results. Most importantly, the singularity curve of the scatterer, which is the most prominent feature of the cartoon model, is obtained with decent precision. Interestingly, the shearlet regularization also requires the least number of iterations in this example.

The reason for the superior performance of the regularization by the shearlet transform is also visible in Figure \ref{fig:Experiment2}.
All three methods handle the singularity curve fairly well, although, naturally, the error is the largest at points where the singularity
is most pronounced, i.e., the upper and lower right corners as well as the middle of the left edge of the centered square. Away from the
singularities, the shearlet regularization yields a far better approximation than the other two approaches, since it is designed to deal very well with smooth regions.

\section{The Linearized Scattering Problem of the Schr\"{o}dinger equation}\label{sec:electro}

The second numerical approach to inverse scattering problems which we consider is that of first linearizing the problem. This
technique is commonly known as Born approximation.
The inverse problem we will study with this method is \emph{quantum mechanical scattering}. After introducing the
inverse problem, our goal will be to provide a theoretical basis for the application of shearlet frames. As before,
we base our considerations on the premise that edges, i.e., curve-like singularities, are the governing features
of the scatterer leading to the cartoon model as appropriate choice.

Once we step away from the nonlinear situation and introduce a linearization, then this argument may, however, not be valid anymore. It could be possible, that linearizing the inverse problem introduces a smoothing that erases all edge-like structures. Fortunately, it has been shown, see, e.g., \cite{OlaPS2001,Serov2013}, that when using a linearization via the inverse Born approximation, certain singularities of the scatterer prevail.

It turns out, that finding the inverse Born approximation is a problem of reconstructing functions from possibly corrupted and maybe limited Fourier data. Problems of this sort have been studied when the object under consideration is sparse in some basis or frame, see \cite{Adcock13breakingthe, 2DWaveletRec}.

This gives a first indication, that methods involving shearlets may be appropriate in a regularization of the inverse scattering problem with the Schr\"{o}dinger equation. However, all results on the
regularity of the inverse Born approximation in the literature describe the global regularity in the sense of weak
differentiability. In the case of cartoon-like functions, however, we  have strong local, but poor global regularity. To be able to exploit shearlets in the context of this problem, in this section, we will prove a local regularity result for the inverse Born approximation.

\subsection{The Inverse Problem}\label{sec:TheInverseProblem}

The \emph{time-independent Schr\"odinger equation}, for $f \in L^2(\R^2)$, a wave number $k>0$, and with $u = u^s + u^{\textrm{inc}}_d$ as in \eqref{eq:solution_ud}, is given by
\begin{eqnarray} \label{eq:SchrEq}
 \Delta u + (f + k^2) u &=& 0,\\
 \lim \limits_{r\to \infty} r^{\frac{1}{2}}\left(\frac{\partial u^s(x)}{\partial r} - ik u^s(x) \right)  &=& 0.
 \nonumber
\end{eqnarray}
This equation describes the motion of a single particle moving in an electric field with potential $f$. Therefore $f$ should be smooth, whenever the permittivity of the medium does not change drastically, and the solution can jump when the medium and the permittivity changes.

For $\theta \in [0,\pi]$ and $\tau_\theta = (\cos(\theta), \sin(\theta))$, the associated \emph{backscattering amplitude} is defined by
\[
 A(k,-\theta,\theta) = \int_{\R^2} e^{i k \left \langle \tau_\theta, y \right \rangle} f(y) u(y)\, dy,
\]
and the inverse problem is to reconstruct the potential $f$ from $A$. 

Using  $|\xi| \tau_\theta := \xi$ to denote  polar
coordinates, the  \emph{Born approximation} of the solution $f$ is defined as the inverse Fourier transform of
the function $\xi \mapsto A(k,-\theta,\theta)$, given by
\begin{equation}
 f_B(x) = \int_{\R^2} e^{-i\left \langle \xi, x \right \rangle} A(\frac{|\xi|}{2}, \theta, -\theta) \, d\xi.\label{eq:DefOfFB}
\end{equation}


%
Applying the Lippmann-Schwinger integral equation \eqref{eq:LippmannSchwinger} iteratively as  in \cite{OlaPS2001}, yields that
\begin{equation}
 f_B(x) = \sum_{j=1}^m q_j(x) + q_{m+1}^R(x),\label{eq:iteration}
\end{equation}
with
\begin{eqnarray}
q_j(x) &=& \frac{1}{4\pi^2} \intRtwo \intRtwo e^{i\left \langle \xi, x + \frac{y}{2}\right \rangle} f(y)
(\mathcal{G}_{|\xi|})^j(e^{i\left \langle \xi, \cdot \right \rangle})(y)\, dy\, d\xi, \quad 1 \le j \le m,\nonumber\\
q_{m+1}^R(x) &=& \frac{1}{4\pi^2} \intRtwo \intRtwo e^{i\left \langle \xi, x + \frac{y}{2}\right \rangle} f(y)
(\mathcal{G}_{|\xi|})^{m+1}(u(\cdot, |\xi|, \hat{\xi}))(y)\, dy\, d\xi,\label{qj}
\end{eqnarray}
where $\mathcal{G}_{|\xi|}: L^2_{loc}(\R^2) \to L^2_{loc}(\R^2)$ is the integral operator with kernel
$G_{\frac{{|\xi|}}{2}}(x,y)f(y)$ as defined in \eqref{eq:defi_G}.

\subsection{Local Regularity of the inverse Born Approximation}\label{subsec:Schr_Thr}

To determine the local regularity of the Born approximation $f_B$ from a given potential $f$, we invoke the
representation \eqref{eq:iteration} and observe that we can equally well examine the regularity of the
functions $q_1, \ldots, q_m$, and $ q_{m+1}^R$. To analyze the regularity of the functions $q_j$, we will make use
of the \emph{Radon transform} and the \emph{Projection Slice Theorem}, see e.g. \hl{\cite{GelGV1966genFunctRadonTrafo}}.
%
\begin{definition}
Let $\theta \in [0, \pi)$. Then the \emph{Radon transform} $\cR_\theta$ of a function $f\in L^2(\R^2)$ along a ray $\Delta_{t, \theta} =
\{x \in \R^2: x_1 \cos(\theta) + x_2 \sin(\theta) = t \}$ for $t \in \R$  is defined as
\[
\cR_\theta f(t) := \int_{\Delta_{t, \theta}} f(x)\, ds = \int \int f(x)\delta_0(\left \langle x, \tau_\theta\right \rangle  - t)\, dx.
\]
\end{definition}

\begin{theorem}[Projection Slice Theorem]
\label{theo:projectionslice}
Let $f\in L^2(\R^2)$. Then, for all $\theta \in [0, \pi)$ and $\xi \in \R$,
\[
\widehat{\cR_\theta f}(\xi) = \widehat{f}(\xi \cos(\theta), \xi \sin(\theta)).
\]
Here $\widehat f$ denotes the Fourier transform.
\end{theorem}
%

Before we state and prove the local regularity of the functions $q_j$, $j = 1, \ldots, m$, we fix some notation. We will denote the Sobolev spaces of functions with $s$ weak derivatives in $L^2(\R^2)$ by $H^s(\R^2)$ and the functions that are locally in $H^s(\R^2)$ by $H^s_{loc}(\R^2)$. Furthermore, $H^s(x)$  is the local Sobolev space of $s$ times weakly differentiable functions with weak derivatives in $L^2(x)$, and $L^2(x)$ the space of distributions that are $L^2$ on a neighborhood of $x$, see \cite{BeaR1984}. Finally,
we denote by $C^{k,\alpha}(\R^2)$ the $k$ times differentiable functions with a H\"older continuous $k$-th derivative with H\"older coefficient $\alpha$, writing $C^{k}(\R^2)$ if the $k$-th derivative is simply continuous.

Then we have the following regularity result.
\begin{theorem}\label{thm:MainSchroedinger}
Let $\epsilon >0$, let $s \in \N, s \geq 2$, and let, for some $x_0\in \R^2$, $f\in L^2(\R^2) \cap H^{s+\epsilon}(x_0)$  be compactly supported and real valued. Then the $q_j$ defined in (\ref{qj}) satisfy $q_j\in H^s(x_0)$ for all $j\in \N$, and, in particular, $f_B \in H^s(x_0)$.
\end{theorem}
\begin{proof}
We start by proving the local regularity of $q_1$.
Let $x_0 \in \R^2$ be such that $f \in  H^s_{\mathrm{loc}}(x_0)$, then we aim to prove that
\[
q_1(x) =  \frac{1}{4\pi^2} \intRtwo \intRtwo \intRtwo e^{i\left \langle \xi, x + \frac{y}{2} + \frac{z}{2}\right \rangle} f(y)
\mathcal{G}_{|\xi|}(y,z) f(z)\, dz\, dy\, d\xi \in H^s(x_0).
\]
For this, we introduce a smooth cutoff function $\phi$ supported in a neighborhood $U_{x_0}$ of $x_0$, where $f$ is $H^{s+\epsilon}$ such that $\phi \equiv 1$ on a strictly smaller neighborhood of $x_0$. For $x \in U_{x_0}$, the function $\phi$ is now used to decompose $q_1$ as
\begin{align*}
q_1(x) = & \ \frac{1}{4\pi^2} \intRtwo \intRtwo \intRtwo e^{i\left \langle \xi, x + \frac{y}{2} + \frac{z}{2}\right \rangle} \phi(y) f(y) \mathcal{G}_{|\xi|}(y,z) f(z)\, dz\, dy\, d\xi\\
	+&\  \frac{1}{4\pi^2} \intRtwo \intRtwo \intRtwo e^{i\left \langle \xi, x + \frac{y}{2} + \frac{z}{2}\right \rangle} (1-\phi)(y) f(y) \mathcal{G}_{|\xi|}(y,z) f(z)\, dz\, dy\, d\xi\\
	=: & \ \mathcal{I}_1(x) + \mathcal{I}_2(x).
\end{align*}
We further decompose $\mathcal{I}_2$ as 
\begin{align*}
  \mathcal{I}_2(x) =&\  \frac{1}{4\pi^2} \intRtwo \intRtwo \intRtwo e^{i\left \langle \xi, x + \frac{y}{2} + \frac{z}{2}\right \rangle} (1-\phi)(y) f(y) \mathcal{G}_{|\xi|}(y,z) \phi(z)f(z)\, dz\, dy \, d\xi\\
  + &\ \frac{1}{4\pi^2} \intRtwo \intRtwo \intRtwo e^{i\left \langle \xi, x + \frac{y}{2} + \frac{z}{2}\right \rangle} (1-\phi)(y) f(y) \mathcal{G}_{|\xi|}(y,z) (1-\phi)(z)f(z)\, dz\, dy\, d\xi\\
  = :& \ \mathcal{I}_{2,1}(x) + \mathcal{I}_{2,2}(x).
\end{align*}
and study each of the integrals $\mathcal{I}_1$, $\mathcal{I}_{2,1}$, and $\mathcal{I}_{2,2}$
separately.

\emph{Regularity of $\mathcal{I}_1$}: We use a representation from \cite[Lem. 1.1]{OlaPS2001}, which yields, for $x \in U_{x_0}$,
\begin{align}
 \mathcal{I}_{1}(x) &= \intRtwo \intRtwo \frac{\mathcal{F}(f)(\xi) \mathcal{F}(\phi f)(\eta)}{\left \langle \eta, \xi \right \rangle - i0}
e^{ i \left \langle x, \xi + \eta \right \rangle}\, d\eta\, d\xi \nonumber\\
 &=:\int \limits_{|\xi|\geq 1}\intRtwo \frac{\mathcal{F}(f)(\xi) \mathcal{F}(\phi f)(\eta)}{\left \langle \eta, \xi \right \rangle - i0}
e^{ i \left \langle x, \xi + \eta \right \rangle}\, d\eta\, d\xi + \nu_1(x), \label{eq:noLowFrequencies}
\end{align}
where
\begin{equation} \label{eq:pv}
\left(\left \langle \xi, \eta \right \rangle -i0 \right)^{-1} = \textsf{p.v.}(\left \langle \xi, \eta \right \rangle)^{-1}-\pi i
\delta_0(\left \langle \xi, \eta \right \rangle),
\end{equation}
with $\textsf{p.v.}$ denoting the Cauchy principal value as stated in \cite[Lem. 1.2]{OlaPS2001}.
 We obtain that $\nu_1\in C^\infty(\R^2)$ by the following argument. Since
 $$\nu_1 =  \frac{1}{4\pi^2} \int \limits_{|\xi|\leq 1} \intRtwo \intRtwo e^{i\left \langle \xi, x + \frac{y}{2} + \frac{z}{2}\right \rangle} f(y) \mathcal{G}_{|\xi|}(y,z) \phi(z)f(z)\, dz\, dy \, d\xi$$
we can employ the compact support of $f$ and an estimate of the form:
$$\| \mathcal{G}_{|\xi|}(f)\|_{L^p} \leq C_p |\xi|^{2(1/q-1/p)-2}\|f\|_{L^q},$$
with $p=6, q=5/6$, which can be found in \cite{Syl2009}, to obtain that
$$\xi \mapsto \intRtwo \intRtwo e^{i\left \langle \xi, x + \frac{y}{2} + \frac{z}{2}\right \rangle} f(y) \mathcal{G}_{|\xi|}(y,z) \phi(z)f(z)\, dz\, dy \in L^1(\R^2).$$
Since we only take the Fourier transform over a bounded frequency part, we obtain that $\nu_1$ is analytic.

To show the local regularity of \eqref{eq:noLowFrequencies} as a function of $x$, we introduce the symbol
\[
 a(x, \xi): = \intRtwo \frac{\mathcal{F}(\phi f)(\eta)}{\left \langle \eta, \xi \right \rangle - i0} e^{i \left \langle x, \eta \right \rangle}  d\eta,
\quad (x ,\xi) \in \R^2 \times \R^2.
\]
It follows from \cite[Thm. 1.3]{BeaR1984}, that if the function
\[
g_\xi : U_{x_0} \to \C, \quad g_\xi(x) := a(x, \xi), \quad \xi \in \R^2,
\]
is in $H^s(x)$ for some $s\in \N$, $s >1$ and if $\|g_\xi\|_{H^s(x_0)}$ is independent of $\xi$, then the corresponding pseudo-differential operator $a(x,D)$ is a bounded operator from $H^s(x_0)$ to $H^s(x_0)$. By definition, we have
\begin{equation}\label{eq:aInu}
a(x,D)(f) = \mathcal{I}_{1}(x) - \nu_1(x).
\end{equation}
Thus it remains to prove that, for any $|\xi|\geq 1$, we have $g_\xi \in H^s(x_0)$ and $\|g_\xi\|_{H^s(x_0)}$ is independent of $\xi$. By \eqref{eq:pv}, we obtain
\begin{equation}
a(x, \xi) = \textsf{p.v.}\intRtwo \frac{\mathcal{F}(\phi f)(\eta)}{\left \langle \eta, \xi \right \rangle} e^{i \left \langle x, \eta \right \rangle}\,  d\eta - \pi i \int \limits_{\left \langle \eta, \xi \right \rangle = 0} \mathcal{F}(\phi f)(\eta) e^{i \left \langle x, \eta \right \rangle}\,  d\eta. \label{eq:TheSplittingOfi0}
\end{equation}
Since $\phi f \in H^{s+\epsilon}(\R^2)$ is compactly supported, it follows that $\mathcal{F}(\phi f) \in C^\infty(\R^2)$ and
\[
\int_{\R^2} (1+|\xi|^2)^{s+\epsilon}|\mathcal{F}(\phi f)(\eta)|^2 d\eta < \infty.
\]
Passing to polar coordinates yields
\[
 \int_{0}^{2 \pi} \int_{\R^+} r(1+|r|^2)^{s+\epsilon}|\mathcal{F}(\phi f)(r \tau_\theta)|^2 \,dr \,d\theta < \infty,
\]
and by the smoothness of
\begin{equation}\label{eq:theta_function}
\theta \mapsto \int_{\R} r(1+|r|^2)^{s+\epsilon}|\mathcal{F}(\phi f)(r \tau_\theta)|^2 d r,
\end{equation}
the terms $\int_{\R} |r|(1+|r|^2)^{s+\epsilon}|\mathcal{F}(\phi f)(r \tau_\theta)|^2 d r$ are uniformly bounded with
respect to $\theta$. This implies the desired regularity of the second term of $\eqref{eq:TheSplittingOfi0}$ as a
function of $x$.

We continue with the first term of \eqref{eq:TheSplittingOfi0}, i.e., with
\[
\textsf{p.v.}\intRtwo \frac{\mathcal{F}(\phi f)(\eta)}{\left \langle \eta, \xi \right \rangle} e^{i \left \langle x, \eta \right \rangle}\,  d\eta
= \textsf{p.v.} \int \limits_{0}^{\pi} \frac{1}{\left \langle \tau_\theta, \xi \right \rangle} \int\limits_{\R} \mathcal{F}(\phi f)(r \tau_\theta )
e^{i r \left \langle x, \tau_\theta \right \rangle } \,dr \,d\theta. \label{eq:ApplyFourierSliceHere}
\]
By substitution, w.l.o.g. we may assume that $\xi = (1,0)$. Application of Theorem~\ref{theo:projectionslice}
shows that \eqref{eq:ApplyFourierSliceHere} equals
\[
\textsf{p.v.} \int \limits_{0}^{\pi} \frac{1}{\left \langle \tau_\theta, \xi \right \rangle} \cR_\theta (\phi f)(\left \langle x, \tau_\theta \right \rangle)\, d\theta
= \lim \limits_{\epsilon \downarrow 0}  \int \limits_{[0,\frac{\pi}{2}-\epsilon] \cup [\frac{\pi}{2}+\epsilon,\pi]} \frac{1}{\cos(\theta)}
\cR_\theta (\phi f)(\left \langle \cdot, \tau_\theta \right \rangle)\, d\theta.
\]
For a given sequence $\epsilon_n \to 0$ as $n \to \infty$, we set ${E}_n:=[0,\frac{\pi}{2}-\epsilon_n] \cup [\frac{\pi}{2}+\epsilon_n,\pi]$
and show that the sequence
\begin{equation}
 \left(\,\,\int \limits_{{E}_n} \frac{1}{\cos(\theta)} \cR_\theta (\phi f)(\left \langle \cdot, \tau_\theta \right \rangle)\, d\theta\right)_{n \in \N} \label{eq:theCauchySequence}
\end{equation}
is a Cauchy sequence in $H^s(x_0)$.

We first observe that by Theorem \ref{theo:projectionslice}, the finiteness of the integral in
\eqref{eq:theta_function} implies that
\[
\xi \mapsto |\xi|^{s+\frac{1}{2}+\epsilon} \cdot \widehat{\cR_\theta (\phi f)}(\xi) \in L^2(\R).
\]
This yields $\cR_\theta (\phi f) \in H^{s+\frac{1}{2}+\epsilon}(\R)$. By the Sobolev embedding theorem, \cite[Thm 5.4]{Adams1975} we have $H^{s+\frac{1}{2}+\epsilon}(\R) \hookrightarrow C^{s,\epsilon}(\R)$,
which implies that $\cR_\theta (\phi f) (\left \langle \cdot, \tau_\theta\right \rangle) \in C^{s,\epsilon}(\R^2)$. Hence for a multi-index $\gamma$ with
$|\gamma| \leq s$, taking the $\gamma$th derivative of each element of \eqref{eq:theCauchySequence} yields
\begin{equation} \label{eq:Cauchy2}
 \left(\,\,\int \limits_{\mathcal{E}_n} \frac{1}{\cos(\theta)} D^{\gamma} (\cR_\theta (\phi f))(\left \langle \cdot, \tau_\theta \right \rangle) d\theta\right)_{n \in \N}.
\end{equation}
To prove that this sequence is a Cauchy sequence in $L^2(x_0)$ we show that $\theta \mapsto D^{\gamma}\cR_\theta (\phi f)(\left \langle \cdot, \tau_\theta \right \rangle)$
is H\"older continuous on $[0,\pi)$. In fact, by the chain rule and Theorem \ref{theo:projectionslice}, for
$\theta, \theta' \in [0,\pi)$, we have
\begin{eqnarray}\nonumber
\lefteqn{\left|D^\gamma (\cR_\theta (\phi f))(\left \langle \cdot, \tau_\theta \right \rangle)
- D^\gamma (\cR_{\theta'} (\phi f))(\left \langle \cdot, \tau_{\theta'} \right \rangle) \right|}\\ \nonumber
& \leq & C \cdot \left(\left|(\cR_{\theta'} (\phi f))^{(|\gamma|)}(\left \langle \cdot, \tau_{\theta'} \right \rangle)
- (\cR_{\theta'} (\phi f))^{(|\gamma|)}(\left \langle \cdot, \tau_{\theta}\right \rangle)\right|\right.\\
& & \hspace*{0.7cm}+ \left.\left|(\cR_{\theta'} (\phi f))^{(|\gamma|)}(\left \langle \cdot, \tau_{\theta} \right \rangle)
- (\cR_\theta (\phi f))^{(|\gamma|)}(\left \langle \cdot, \tau_{\theta}\right \rangle)\right|\right).\label{eq:twoTypesOfHoelderCont}
\end{eqnarray}
Since $\cR_\theta (\phi f) \in C^{s,\epsilon}(\R)$, we obtain that the first term of \eqref{eq:twoTypesOfHoelderCont} is bounded by
$C_0 |\tau_\theta-\tau_{\theta'}|^{\alpha}$ for some $\alpha \leq \epsilon$ and $C_0 > 0$. Hence also the local $L^2(x_0)$-norm
of the first term is bounded by $C_1 |\tau_\theta-\tau_{\theta'}|^{\alpha}$ with a possibly different constant $C_1$. In the sequel, $C_\nu$, $\nu \in \N$
will always denote a positive constant.

To estimate the $L^2(x_0)$-norm of the second term of \eqref{eq:twoTypesOfHoelderCont}, it suffices to show
\begin{equation} \label{eq:longproof1}
\|(\cR_\theta (\phi f))^{(|\gamma|)} - (\cR_{\theta'} (\phi f))^{(|\gamma|)}\|_{L^2(\R)} \leq C_2 |\tau_\theta - \tau_{\theta'}|^\alpha \quad \mbox{for some } 0< \alpha < 1/2.
\end{equation}
Using the Plancherel identity \cite{Mall2009SignalProcessing} and Theorem~\ref{theo:projectionslice}, we obtain
\begin{eqnarray}\nonumber
\lefteqn{\frac{1}{|\tau_\theta -\tau_{\theta'} |^\alpha} \left\|  \left((\cR_\theta (\phi f))^{(|\gamma|)} - (\cR_{\theta'} (\phi f))^{(|\gamma|)}\right) \right\|_{L^2(\R)}}\\
&=& \frac{1}{2\pi}\left\|\frac{(i \, \cdot)^{|\gamma|}}{|\tau_\theta -\tau_{\theta'} |^\alpha}
 \left(\mathcal{F}(\phi f)(\cdot \, \tau_{\theta})-\mathcal{F}(\phi f)(\cdot \, \tau_{\theta'}) \right)\right\|_{L^2(\R)}\nonumber\\
&\leq& \frac{1}{2\pi}\left\|(i \, \cdot)^{|\gamma|+\alpha} \frac{\mathcal{F}(\phi f)(\cdot \, \tau_{\theta})
-\mathcal{F}(\phi f)(\cdot \,  \tau_{\theta'})}{|\cdot \,  \tau_\theta - \cdot \,  \tau_{\theta'} |^\alpha} \right\|_{L^2(\R)}. \label{eq:HoelderEstimateA}
\end{eqnarray}
For $|\tau_\theta - \tau_{\theta'}|$ small enough, we now pick any multiindex $\rho$ with $|\rho|= |\gamma|$ and $|\tau_\theta^\rho|, |\tau_{\theta'}^\rho|\geq \left(\frac{1}{2}\right)^{|\gamma|}$. Thus, by  \eqref{eq:HoelderEstimateA},
\begin{eqnarray}\nonumber
\lefteqn{\frac{1}{|\tau_\theta -\tau_{\theta'} |^\alpha} \left\|  \left((\cR_\theta (\phi f))^{(|\gamma|)} - (\cR_{\theta'} (\phi f))^{(|\gamma|)}\right) \right\|_{L^2(\R)}}\\
& \le & \frac{1}{2\pi}\left\|(i\, \cdot )^{\alpha} \frac{\tau_\theta^{-\rho}\mathcal{F}(D^\rho(\phi f))(\cdot \,  \tau_{\theta})- \tau_{\theta'}^{-\rho}
\mathcal{F}(D^\rho (\phi f))(\cdot \, \tau_{\theta'})}{|\cdot \, \tau_\theta - \cdot \, \tau_{\theta'} |^\alpha} \right\|_{L^2(\R)}\nonumber\\
& \leq &\frac{1}{2\pi}\left\|(i\, \cdot)^{\alpha} \tau_\theta^{-\rho} \frac{\mathcal{F}(D^\rho(\phi f))(\cdot \, \tau_{\theta})
- \mathcal{F}(D^\rho (\phi f))(\cdot \, \tau_{\theta'})}{|\cdot \, \tau_\theta - \cdot \, \tau_{\theta'} |^\alpha} \right\|_{L^2(\R)} \nonumber\\ 
& &\qquad + \frac{1}{2\pi}\left\|(i \, \cdot)^{\alpha} \frac{\tau_\theta^{-\rho}\mathcal{F}(D^\rho(\phi f))(\cdot \, \tau_{\theta'})
- \tau_{\theta'}^{-\rho}\mathcal{F}(D^\rho (\phi f))(\cdot \, \tau_{\theta'})}{|\cdot \, \tau_\theta - \cdot \, \tau_{\theta'} |^\alpha}
\right\|_{L^2(\R)}\nonumber\\
&=:& {\mathcal M}_1+{\mathcal M}_2\label{eq:theSecondPartForHoelder}.
\end{eqnarray}
Since $\phi f$ has compact support, also $D^\rho(\phi f)$ is compactly supported. Hence, its Fourier transform is H\"older continuous and obeys
\[
\frac{\mathcal{F}(D^\rho(\phi f))(r \tau_{\theta})- \mathcal{F}(D^\rho (\phi f))(r \tau_{\theta'})}{|r \tau_\theta - r \tau_{\theta'} |^\alpha} < h(r \tau_\theta),
\quad\mbox{for all } r \in \R, \, \theta \in [0,\pi),
\]
for a function $h\in L^2(\R^2)$. Thus, the first term $\mathcal M_1$ in \eqref{eq:theSecondPartForHoelder} is bounded, if $\alpha < \frac{1}{2}$. To estimate the second term $\mathcal M_2$ in \eqref{eq:theSecondPartForHoelder}, we
observe, that
\[
\frac{1}{\tau_\theta^\rho} - \frac{1}{\tau_{\theta'}^\rho} = \frac{\tau_\theta^\rho-\tau_{\theta'}^\rho}{\tau_\theta^\rho \tau_{\theta'}^\rho} \leq C_3 |\tau_\theta - \tau_{\theta'}|.
\]
Thus, the term $\mathcal M_2$ is bounded by $C_4 | \tau_\theta - \tau_{\theta'}|^{1-\alpha}$, and we have proved \eqref{eq:longproof1}. Using the estimates for the two terms in \eqref{eq:twoTypesOfHoelderCont} yields that
\[
\|D^\gamma (\cR_\theta (\phi f))(\left \langle \cdot, \tau_\theta \right \rangle)
- D^\gamma (\cR_{\theta'} (\phi f))(\left \langle \cdot, \tau_{\theta'} \right \rangle) \|_{H^s(x_0)}
< C_5 |\tau_\theta - \tau_{\theta'}|^\alpha.
\]
Returning to the sequence in \eqref{eq:Cauchy2}, for $m>n$, we have the estimate
{\allowdisplaybreaks
\begin{align}\nonumber
&\|\int \limits_{\frac{\pi}{2}-\epsilon_n}^{\frac{\pi}{2}-\epsilon_m} \frac{1}{\cos(\theta)} (\cR_\theta (\phi f))(\left \langle \cdot, \tau_\theta \right \rangle) d\theta
+ \int \limits_{\frac{\pi}{2}+\epsilon_m}^{\frac{\pi}{2}+\epsilon_n} \frac{1}{\cos(\theta)} (\cR_\theta (\phi f))(\left \langle \cdot, \tau_{\theta} \right \rangle) d\theta\|_{H^s(x_0)}\\
\nonumber
=& \|\int \limits_{\frac{\pi}{2}-\epsilon_n}^{\frac{\pi}{2}-\epsilon_m} \frac{1}{\cos(\theta)} (\cR_\theta (\phi f))(\left \langle \cdot, \tau_{\theta} \right \rangle) d\theta
- \int \limits_{\frac{\pi}{2}-\epsilon_n}^{\frac{\pi}{2}-\epsilon_m} \frac{1}{\cos(\theta)} (\cR_{\pi-\theta}(\phi f))(\left \langle \cdot, \tau_{\pi-\theta} \right \rangle) d\theta\|_{H^s(x_0)}\\ \nonumber
=& \int \limits_{\frac{\pi}{2}-\epsilon_n}^{\frac{\pi}{2}-\epsilon_m} \frac{1}{\cos(\theta)} \left\|(\cR_\theta (\phi f))(\left \langle \cdot, \tau_{\theta} \right \rangle)
- (\cR_{\pi-\theta}(\phi f))(\left \langle \cdot, \tau_{\pi-\theta} \right \rangle)\right\|_{H^s(x_0)} d\theta \\
\leq& \ C_5 \int \limits_{\epsilon_m}^{\epsilon_n} \frac{|\tau_\theta - \tau_{\pi-\theta}|^{\alpha}}{\cos(\theta)}  d\theta\leq C_6 \int \limits_{\epsilon_m}^{\epsilon_n} \frac{|\pi -2\theta|^{\alpha}}{\cos(\theta)}  d\theta. \label{eq:thisconvergesTozero}
\end{align}
}
Since $\int_{0}^{\pi} |\pi-2\theta|^{\alpha}/\cos(\theta)\, d\theta < \infty$, \eqref{eq:thisconvergesTozero} converges to $0$ for $m>n$ and as $n \to \infty$ and hence
\eqref{eq:theCauchySequence} is a Cauchy sequence in $H^s(x_0)$, which implies the required regularity of $g_\xi$. Thus, $a(\cdot,D)(f)$ is
$s-$times weakly differentiable and using \eqref{eq:aInu}, we obtain the required differentiability of $\mathcal{I}_1$.

\emph{Regularity of $\mathcal{I}_{2,1}$}: The proof of  local regularity of $\mathcal{I}_{1}$ can be applied in a similar way to also prove local regularity
of $\mathcal{I}_{2,1}$.

\emph{Regularity of $\mathcal{I}_{2,2}$}: Using the same argument as for $\mathcal{I}_{1}(x)$, we obtain
\begin{align*}
 \mathcal{I}_{2,2}(x) &= \intRtwo \intRtwo \frac{\mathcal{F}((1-\phi)f)(\xi) \mathcal{F}((1-\phi) f)(\eta)}{\left \langle \eta, \xi \right \rangle - i0}
e^{ i \left \langle x, \xi + \eta \right \rangle} \,d\eta \,d\xi\\
 &=\int \limits_{|\xi|\geq 1}\intRtwo \frac{\mathcal{F}((1-\phi)f)(\xi) \mathcal{F}((1-\phi) f)(\eta)}{\left \langle \eta, \xi \right \rangle - i0}
e^{ i \left \langle x, \xi + \eta \right \rangle} \,d\eta \,d\xi +\nu_2(x),
\end{align*}
for  $x \in U_{x_0}$, where $\nu_2 \in C^\infty(\R^2)$. In this case the approach as for $\mathcal{I}_{1}(x)$ is not applicable anymore, since $(1-\phi) f$ is not globally $s$-times differentiable and, consequently, its Radon transform does not need to be as well. Thus, it is not
possible to construct a pseudo-differential operator, which is bounded from $H^s(x_0)$ to $H^s(x_0)$.

However, since $(1-\phi) f \in L^2(\R^2)$, using the argument of \eqref{eq:thisconvergesTozero} with $s = 0$ and considering the Radon
transform $\cR_\theta((1-\phi)f)$ instead of $\cR_\theta(\phi f)$ yields that with the symbol $b$ being defined as
\[
 b(x, \xi): = \intRtwo \frac{\mathcal{F}((1-\phi) f)(\eta)}{\left \langle \eta, \xi \right \rangle - i0} e^{i \left \langle x, \eta \right \rangle}  d\eta,
\quad (x ,\xi) \in \R^2 \times \R^2,
\]
the function
\[
h_\xi : U_{x_0} \to \C, \quad h_\xi(x) := b(x, \xi), \quad \xi \in \R^2,
\]
is an $L^2(\R^2)$ function with $\|h_\xi\|_{L^2(\R^2)}$ independent of $\xi$.

Approximating $b$ via a sequence
\[
 b^M(x, \xi): = \int \limits_{|\eta| \leq M} \frac{\mathcal{F}((1-\phi) f)(\eta)}{\left \langle \eta, \xi \right \rangle - i0}
e^{i \left \langle x, \eta \right \rangle}  d\eta, \quad \text{for } M\in \N,
\]
we have that for fixed $\xi$, the function $x \mapsto b^M(x, \xi)$ is $C^{\infty}$ on a neighborhood of $x_0$. Hence, $b^M(x, D)$ is a
bounded operator from $H^t(x_0)$ to $H^t(x_0)$ for all $t\geq 0$. In particular, since $(1-\phi)f \equiv 0$ on a neighborhood of $x_0$, we obtain
\[
b^M(x, D)((1-\phi)f) \equiv 0 \quad \mbox{on a neighborhood of }x_0,
\]
which can be chosen to be the same for all $M$. Then
\[
\|b^M(\cdot, \xi) - b(\cdot, \xi)\|_{L^2(\R^2)} \to 0 \quad \mbox{as } M\to \infty \mbox{ uniformly in }\xi,
\]
and hence
\[
\|b^M(x, D)((1-\phi)f) - b(x, D)((1-\phi)f)\|_{L^2(\R^2)} \to 0 \quad \mbox{as }M\to \infty.
\]
In particular, since $b^M(x, D)((1-\phi)f) = 0$ on a neighborhood of $x_0$, also $b(x, D)((1-\phi)f) = 0$ on a neighborhood
of $x_0$. Since $\mathcal{I}_2$ equals $b(x, D)((1-\phi)f)$ up to a smooth function, this yields the claimed regularity of $\mathcal{I}_{2,2}$.

Combining all the terms $\mathcal{I}_{1}$, $\mathcal{I}_{2,1}$, and $\mathcal{I}_{2,2}$ finishes the proof that $q_1 \in H^s(x_0)$.
For the functions  $q_2, \ldots, q_{m}$, using a similar computation as in the proof
of the main theorem of~\cite{Serov2013}, we obtain for  $1 \le j \le m-1$ that
\[
 q_{j+1}(x) = \frac{1}{4\pi^2} \intRtwo \intRtwo \intRtwo e^{i\left \langle \xi, x + \frac{y}{2} + \frac{z}{2}\right \rangle} f(y) \mathcal{G}_{|\xi|}(y,z) q_j(z)\, dz\, dy \, d\xi,
\quad x \in \R^2.
\]
Now we can apply similar arguments as in the proof for $q_1 \in H^s(x_0)$, in particular, splitting $f$ and $q_j$ into two parts and estimating the resulting terms in the same fashion as before.

Finally, to show that $f_B \in H^s(x_0)$, the decomposition \eqref{eq:iteration} indicates that it remains to analyze the regularity
of  $q_{m+1}^R$. It has been shown in \cite[Prop. 4.1]{OlaPS2001}, that $q_{m+1}^R\in H^t(\R^2)$ for all $t<(m+1/2)/2-1$.
Hence choosing $m$ large enough yields the final claim.
\end{proof}

\begin{remark}{\rm
Observe that in Theorem \eqref{thm:MainSchroedinger} we locally lose an $\epsilon$ in the derivative for arbitrarily small $\epsilon>0$,
when going from $f\in L^2(\R^2) \cap H^{s+\epsilon}(x_0)$ to $f_B \in H^s(x_0)$. Certainly, one might ask whether this is in fact necessary.
The examination of $\mathcal{I}_{2,2}$ in the proof of Theorem \eqref{thm:MainSchroedinger} suggests that the regularity of $f_B$ depends only on the term
\[
\intRtwo \intRtwo \intRtwo e^{i\left \langle \xi, x + \frac{y}{2} + \frac{z}{2}\right \rangle} \phi(y) f(y) \mathcal{G}_{|\xi|}(y,z) \phi(z)f(z) \,dz \,dy\, d\xi.
\]
A careful review of the methods of \cite{OlaPS2001} and \cite{Serov2013} seems to indicate that this term should be even smoother
than $\phi(y) f(y)$. Hence we believe that Theorem \ref{thm:MainSchroedinger} can be improved in the sense that locally the regularity
of the Born approximation $f_B$ is higher than the regularity of the contrast function $f$.}
\end{remark}

We now turn to the question of how the Born approximation affects the regularity of a function $f$
that is modeled as a cartoon-like
function. It would certainly be highly desirable that $f_B$ is again a cartoon-like function, and we show next that this is indeed almost the case when posing some weak additional conditions to  $f$.

The proof will use both the known results that the inverse Born approximation does not introduce a global smoothing, see \cite{OlaPS2001,ReyGlobalBorn2007,Serov2013}, as well as Theorem~\ref{thm:MainSchroedinger}, which proves that locally the smoothness does not decrease.
The key point will be that for a scatterer, which is smooth except for some singularity curve, this curve will still be present in the inverse Born approximation. For this result, we introduce the notion of a neighborhood $N_\delta(X)$ of a subset $X \subset \R^2$ defined by $N_\delta(X): = \{x\in \R^2: \inf_{y\in X}\|x-y\|_2 <\delta\}$, where $\delta > 0$.
\begin{corollary}
\label{cor:cartoon}
Let $\epsilon>0$, let $f_0, f_1 \in H^{3+\epsilon}(\R^2)$ be compactly supported, let $D$ be some compact domain with piecewise $C^2$ boundary $\partial D$, and set
\[
f = f_0 +  f_1 \chi_D.
\]
Then, for every $\delta>0$, there exist $\tilde{f}_0^\delta, \tilde{f}_1^\delta \in H^3(\R^2)$ with compact support, $h^\delta \in H^r(\R^2)$
for every $r<\frac{1}{2}$ with $\suppp h^\delta \subset N_\delta(\partial D)$, and $\nu^\delta\in C^\infty(\R^2)$ such that the inverse Born
approximation $f_B$ of $f$ can be written as
\[
f_B = \tilde{f}_0^\delta + \tilde{f}_1^\delta \chi_D + h^\delta + \nu^\delta.
\]
In particular, $f_B$ is a cartoon-like function up to a $C^\infty$ function and an arbitrarily well localized correction term at the boundary.
\end{corollary}
\begin{proof}
Let $f_0, f_1,B$ be as assumed. For a fixed $\delta > 0$, choose $\phi_1,\phi_2, \phi_3 \in C^\infty(\R^2)$ such
that $\phi_1 + \phi_2 + \phi_3 \equiv 1$, $\phi_i \geq 0$ for $i = 1,2,3$, and
\begin{eqnarray*}
 \phi_1 &\equiv& 1 \text{ on } N_\frac{\delta}{2}(\partial D), \quad \suppp \phi_1 \subset  N_\delta(\partial D),\\
 \phi_2 &\equiv& 1 \text{ on } (\suppp f_0 \cup \suppp f_1) \setminus N_\delta(\partial D),\\
 \suppp \phi_2 &\subset&  N_\delta(\suppp f_0
\cup \suppp f_1) \setminus N_\frac{\delta}{2}(\partial D).
\end{eqnarray*}
By Theorem \ref{thm:MainSchroedinger}, it follows that $\phi_2 f_B \in H^3(\R^2)$ and $\phi_3 f_B \in C^\infty(\R^2)$. Then \eqref{eq:iteration} implies that
\[
 \phi_i f_B = \phi_i f + \phi_i q_1^R, \quad \text{ for } i= 1,2,3,
\]
and thus $\phi_2 q_1^R \in H^3(\R^2)$ and $\phi_3 q_1^R \in C^\infty(\R^2)$.

Defining
\[
 \tilde{f}_0^\delta := f_0 + \phi_2 q_1^R,\quad \tilde{f}_1^\delta := f_1,\quad h^\delta := \phi_1 q_1^R,\quad \mbox{and} \quad \nu^\delta := \phi_3 q_1^R,
\]
then the Sobolev embedding theorem~\cite{Adams1975}, implies that $\tilde{f}_0^\delta, \tilde{f}_1^\delta\in C^2(\R^2)$. Then by \cite[Prop. 4.1]{OlaPS2001} it follows that
$q_1^R \in H^r(\R^2)$ for all $r<\frac{1}{2}$, and hence $h^\delta \in H^r(\R^2)$, and $\suppp h^\delta \subset N_\delta(\partial B)$ follows
by construction. The function $\nu^\delta$ is $C^\infty$, since $\phi_3 q_1^R \in C^\infty(\R^2)$. Thus the main assertion is proved and the 'in particular' part follows immediately.
\end{proof}

\subsection{Numerical Examples}\label{subsec:Schr_Num}

The previously derived results, in particular, Corollary \ref{cor:cartoon} now enable the utilization of common numerical approaches for the linearized problem.
In fact, Corollary \ref{cor:cartoon} implies that the inverse Born approximation of a cartoon-like function will indeed be almost a cartoon-like function. It
thus seems conceivable to assume that the inverse Born approximation can again be sparsely approximated by shearlets, i.e., the sparsity  of the expansion -- which
is key to most regularization approaches -- is maintained during the linearization process. To illustrate why this is indeed the case, let us consider the following
numerical example.

\subsubsection{Sparse Approximation of the Inverse Born Approximation by Shearlets}

For this, we consider the cartoon-like function $f$ given in the top left of Figure \ref{fig:BornExp1}.
Denoting by $ [l;m]$ the vector with entries $l, m \in \Z$, we compute the discrete backscattering amplitude of $f$ from full scattering data:
$$
 A(l,m) = \int_{\R^2} e^{ i \cdot ( [l;m],  y) } f(y) u_{l,m}(y)\, dy \quad \mbox{for all } l, m \in [-1,1]\cap \frac{1}{128}\Z,
$$
where $u_{l,m}$ is the solution of \eqref{eq:SchrEq} with incident wave $u^{inc}(y) = e^{ i \cdot ( [l;m]/\|[l;m]\|,  y) }$ and wavenumber $k = \|[l;m]\|$. By taking the discrete inverse Fourier transform of $A$, we obtain the inverse Born approximation $f_B$ depicted in the top right of Figure \ref{fig:BornExp1}.
It is immediately visually evident, that the reconstruction exhibits the same singularity curve with a smooth distortion, as predicted by Corollary \ref{cor:cartoon}.
\begin{figure}[htb!]
\centering
 \includegraphics[width = 0.32\textwidth]{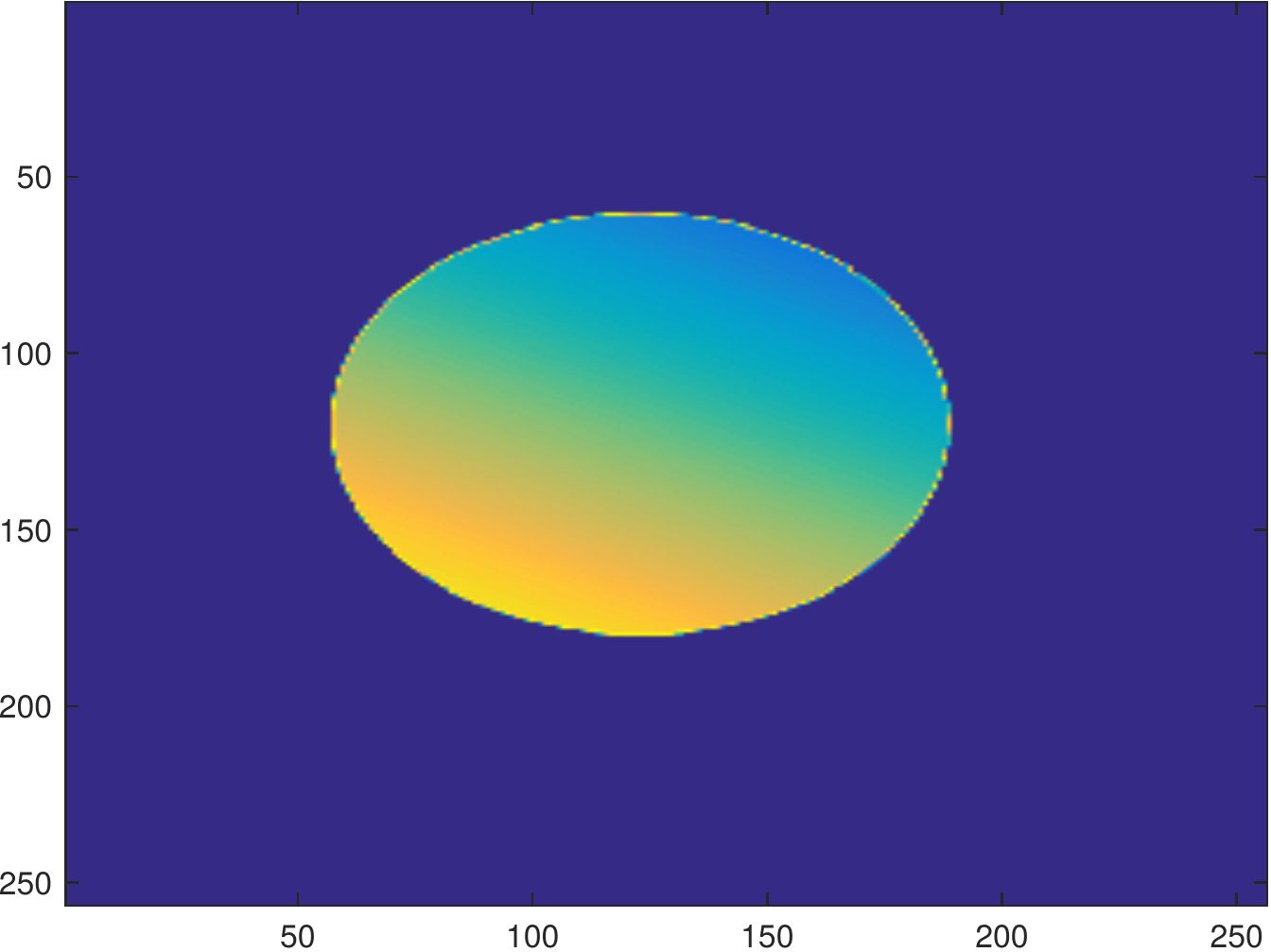} \includegraphics[width = 0.32\textwidth]{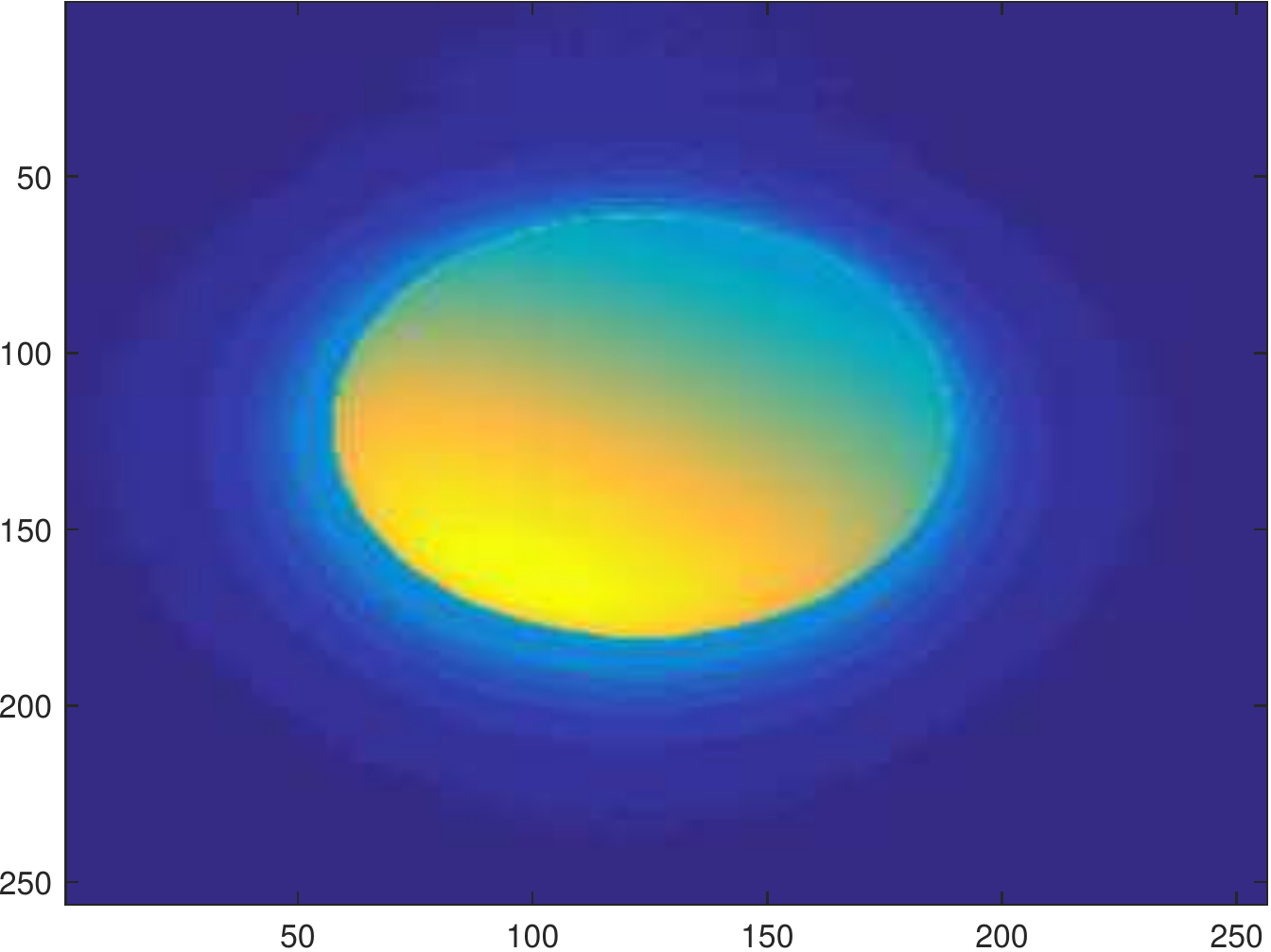}   \includegraphics[width = 0.31\textwidth]{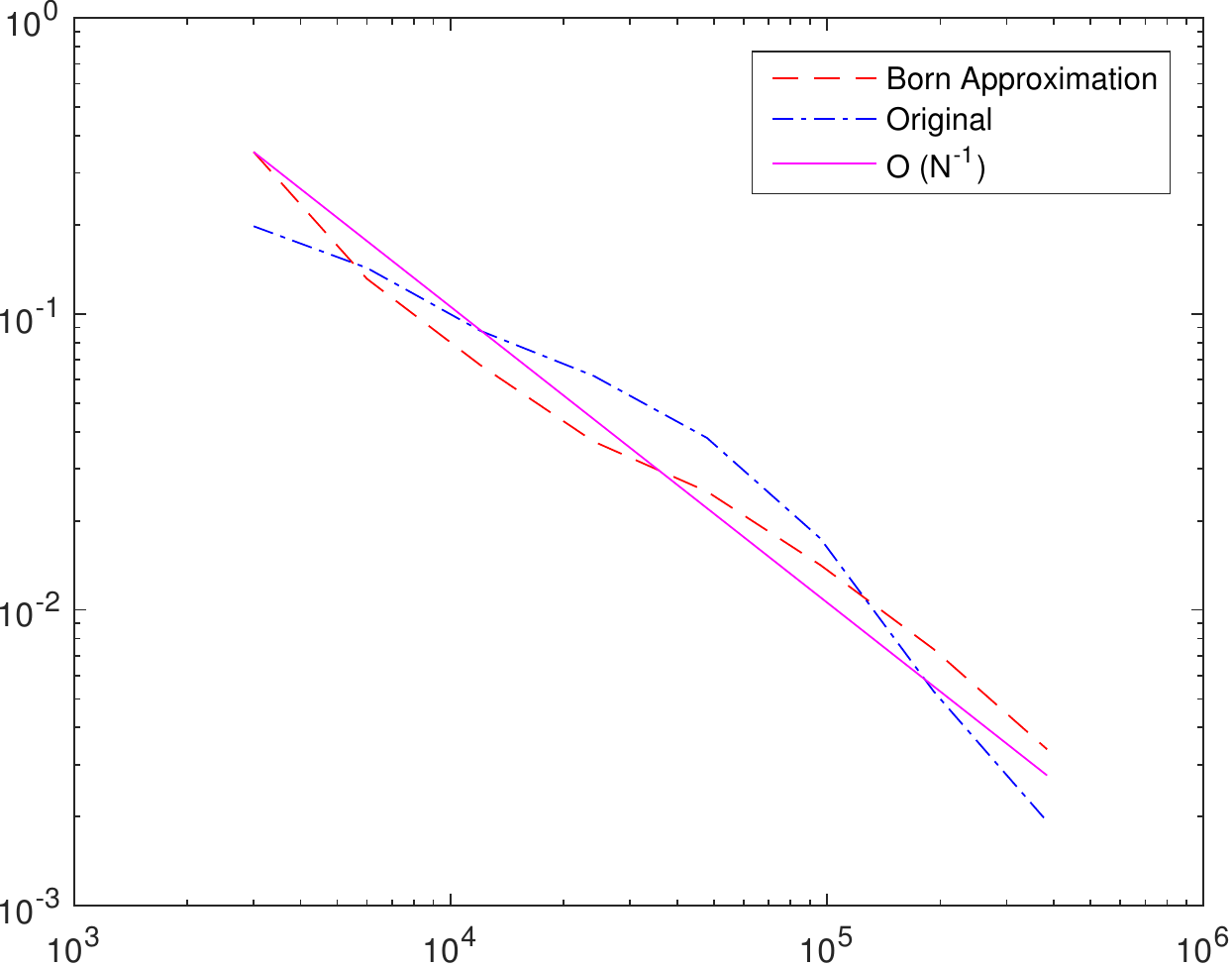}

 \medskip
 \includegraphics[width = 0.24\textwidth]{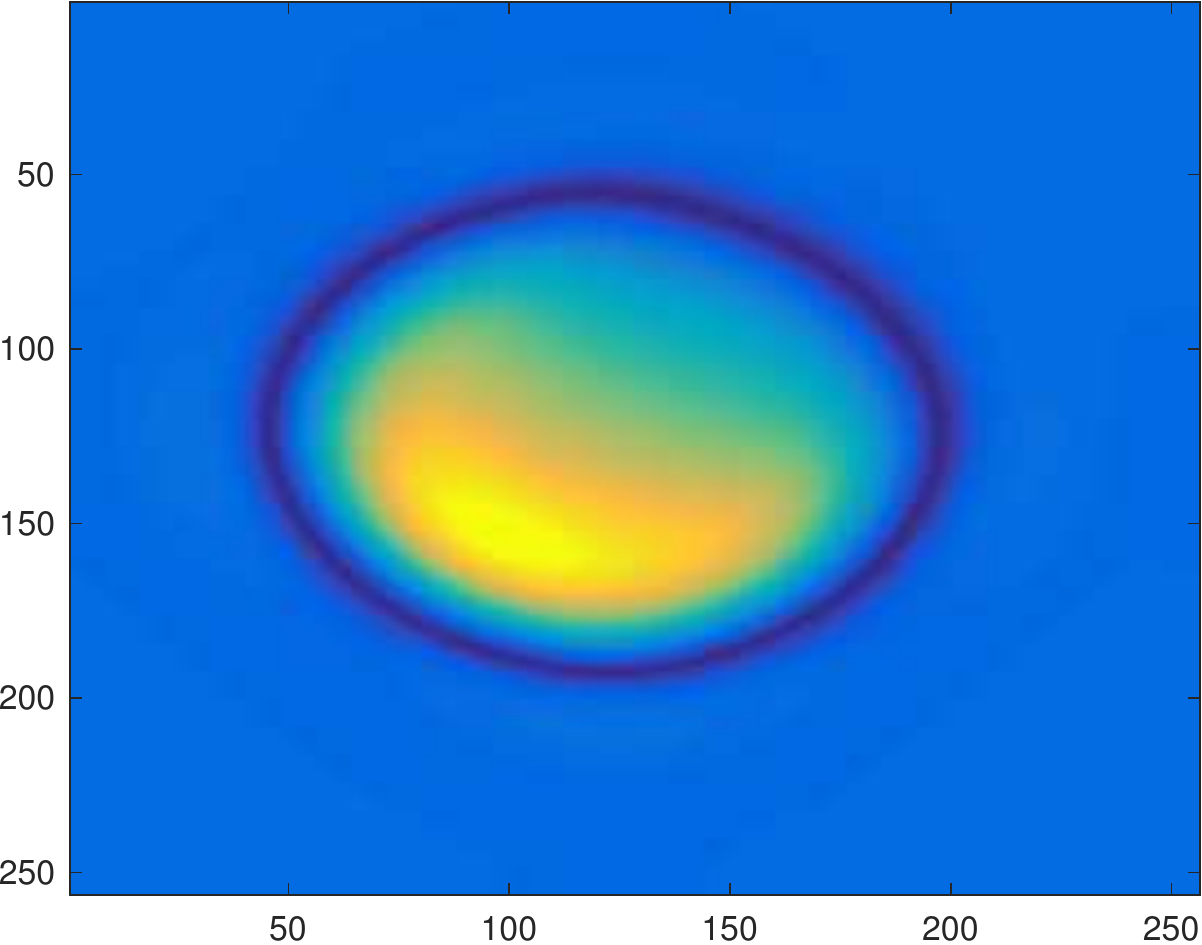} \includegraphics[width = 0.24\textwidth]{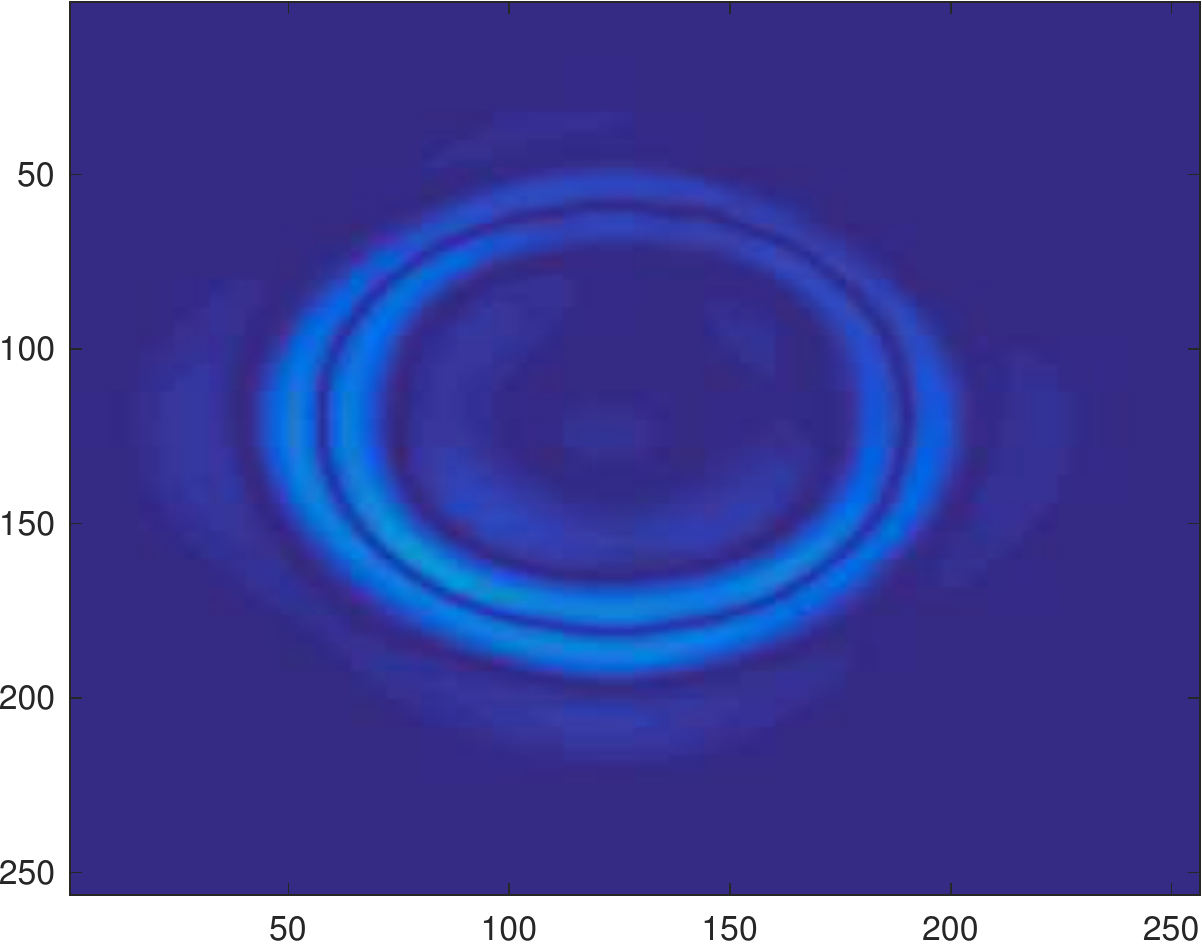} \includegraphics[width = 0.24\textwidth]{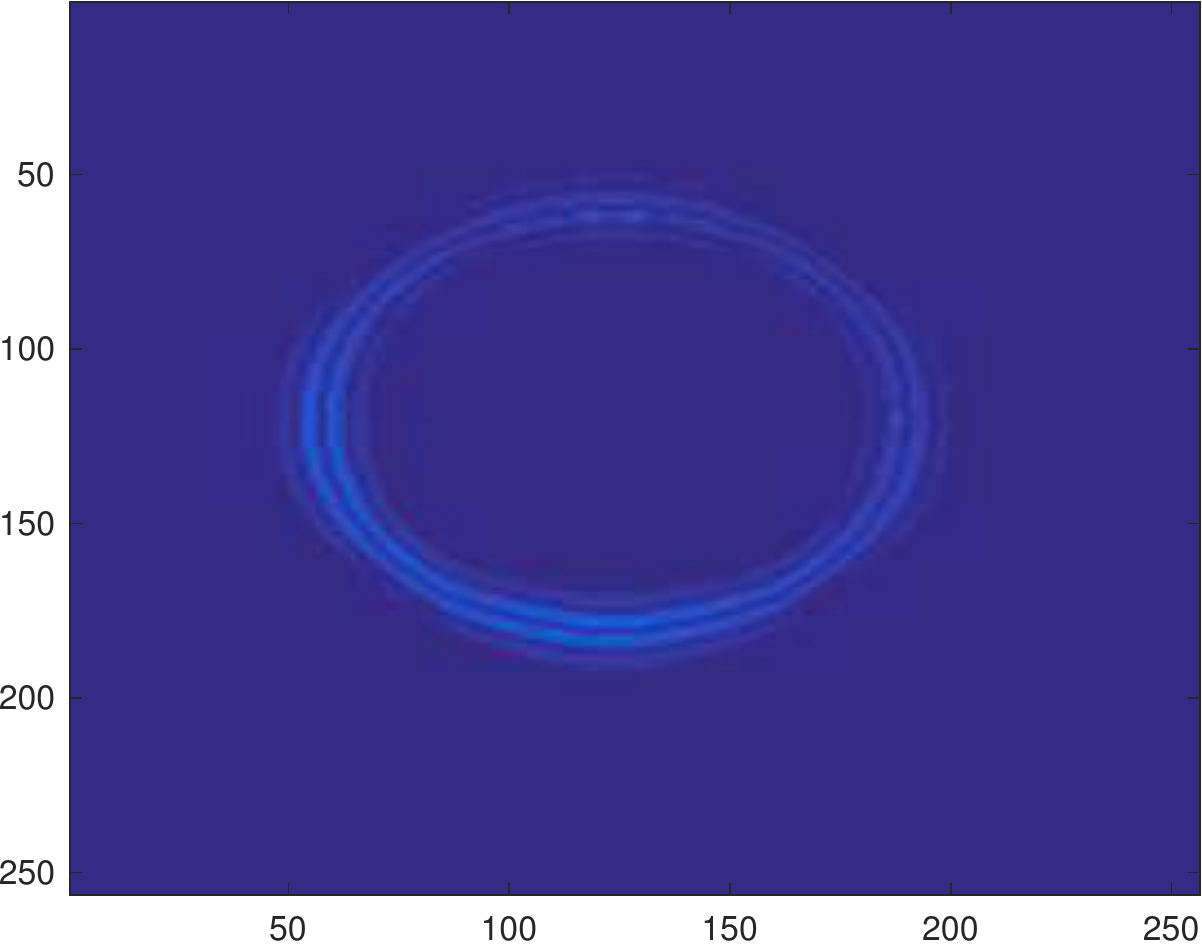} \includegraphics[width = 0.24\textwidth]{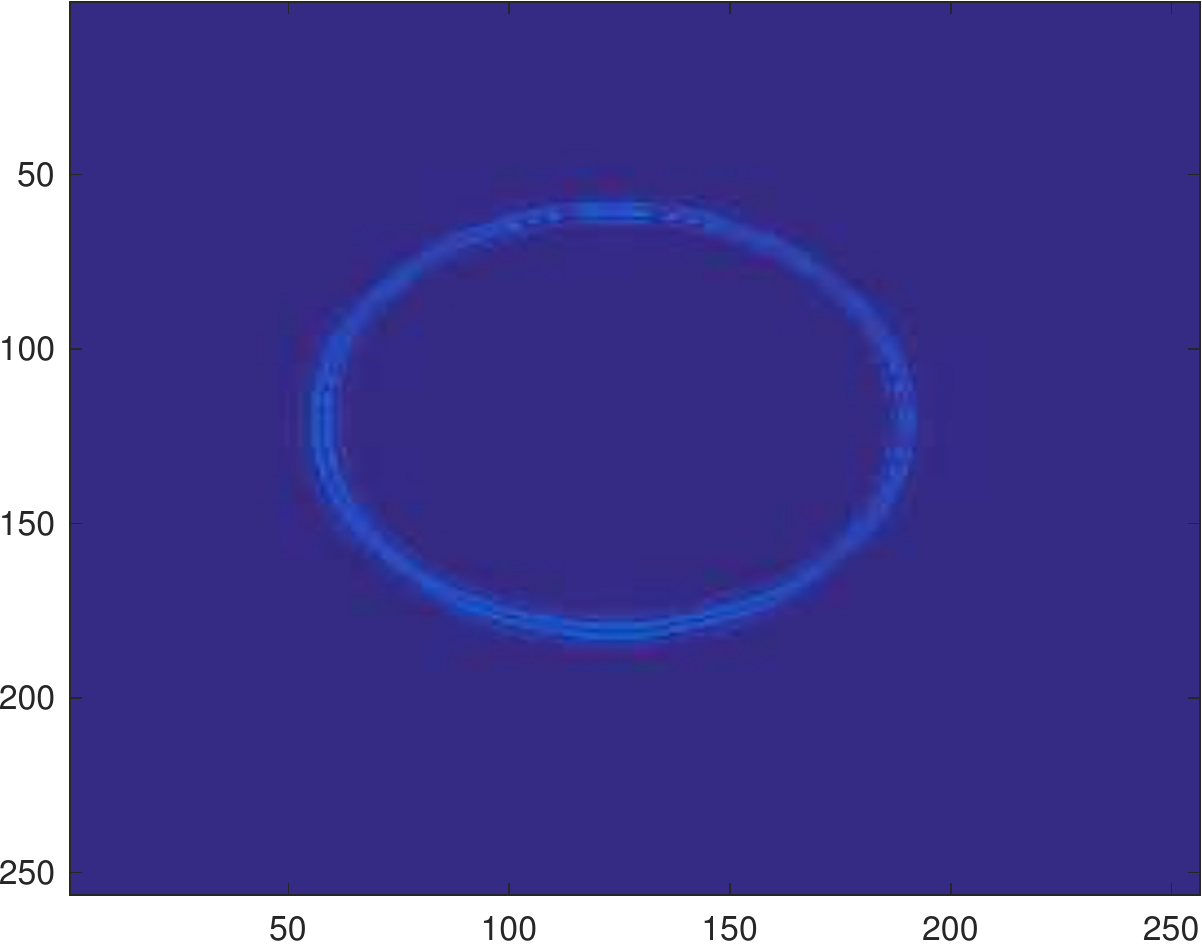}
 \caption{\textbf{Top (from left to right):} \hl{A cartoon-like function $f$, the corresponding inverse Born approximation $f_B$, $N$-term approximation error with the shearlet frame for the original $f$ and its Born approximation $f_B$. The decay is asymptotically of order $O(N^{-1})$.}\\
  \textbf{Bottom (from left to right):} Absolute values of the reconstruction of the inverse Born approximation $f_B$ using only shearlet coefficients on levels $J = 0,1,2,3$. The energy of the higher level shearlet coefficients is clearly concentrated along the jump singularity of the original cartoon-like function.}\label{fig:BornExp1}
 \end{figure}
To examine the sparsity structure of $f_B$, we now compute reconstructions of $f_B$ using only shearlet coefficients on a fixed level. The results are depicted
in Figure~\ref{fig:BornExp1} and show that the energy of the reconstruction using higher levels is concentrated along the discontinuity curve. This provides
us with the first (qualitative) indication that the inverse Born approximation $f_B$ can again be sparsely approximated by shearlets.

Aiming also for quantitative evidence of the ability of shearlets to sparsely approximate $f_B$, we next analyze the \hl{$N-$ term approximation rate by the shearlet system of $f$ and $f_B$} in \hl{the top right of Figure \ref{fig:BornExp1}}.
We observe that both expansions have the same order of decay \hl{of} $O(N^{-1})$ as $N\to \infty$. Thus, the rate of approximation of the
original cartoon-like function and its inverse Born approximation coincide asymptotically with the \hl{optimal rate for cartoon-like functions}. Hence, concluding,
our theoretical analysis (Corollary \ref{cor:cartoon}) that the sparsity structure of cartoon-like functions in the shearlet expansion prevails after
taking the inverse Born approximation becomes also evident in numerical experiments.


\subsubsection{Solution of the Linearized Problem}

Having settled the question of sparse approximation of the inverse Born approximation by shearlets both theoretically and numerically allows us to then
utilize one of the numerous approaches to incorporate sparse regularization in linear inverse problems, see \cite{CandD2002CurveletsInIllPosedProblems,ColEGL2010ShearletsRadon,PEHShearDoconv2009} as well as \cite{KL2012}.

Let us take a closer look at the linearized problem that we face in the case of the inverse scattering problem of the Schr\"{o}dinger equation.
For a given backscattering amplitude $A$, the inverse Born approximation can be obtained by simply taking the inverse Fourier transform of $A$ as described
in \eqref{eq:DefOfFB}. In real world applications the inversion will be more involved, since the whole backscattering amplitude might not be accessible and only
partial measurements can be used. Furthermore, these measurements are likely to be corrupted by noise. At this point the problem becomes a problem of reconstructing
functions exhibiting a known sparsity structure from given partial noisy Fourier measurements.

Now assume that samples of the backscattering amplitude $(A(k_i, \theta_i))_{i = 1}^n =: \mathbf{A}\in \R^n$ are given. Using the synthesis operator $T_{\tilde{\Phi}}$
of the dual of a shearlet frame and the notation of Subsection \ref{sec:TheInverseProblem}, we can then define
%
\[
K: \R^m \to \R^n, \quad c \mapsto K c := \left(\int_{\R^2} e^{i k_i \left \langle \tau_{\theta_i}, y \right \rangle} \left(T_{\tilde{\Phi}}(c)\right)(y)  dy \right)_{i = 1}^n.
\]
Since, due to the limited amount of measurements, this problem is typically underdetermined, the sparsity introduced by the shearlet expansion and guaranteed by
Corollary \ref{cor:cartoon} will be used, and the inversion of $K$ will be casted as the regularized minimization problem
\[
 \argmin_{c\in \R^m} \|K c - \mathbf{A}\|_{\R^n} + \Xi(c),
\]
where $\Xi: \R^m \to \R$ is a suitable sparsity promoting functional. Traditionally, if $\Xi$ is chosen as $\lambda \|. \|_1$ for some $\lambda >0$, this
problem can be solved by an iterative shrinkage thresholding algorithm, \cite{DauDD}. Along these lines, reconstruction problems from Fourier data have
been studied under the premise of sparse representations in \cite{Adcock13breakingthe, 2DWaveletRec} for wavelets and with some numerical examples also
for shearlets in \cite{RomAH2014}. A full analysis of the reconstruction problem with shearlet systems as well as numerical examples is given in \cite{JackMa}.

\section*{Acknowledgements}

The first author acknowledges support by the Einstein Foundation Berlin, by the Einstein Center for Mathematics Berlin
(ECMath), by Deutsche Forschungsgemeinschaft (DFG) Grant KU 1446/14, by the DFG Collaborative Research Center TRR 109
``Discretization in Geometry and Dynamics'', and by the DFG Research Center {\sc Matheon} ``Mathematics for key technologies''
in Berlin. The second author also acknowledges support by {\sc Matheon}, and the third author thanks the DFG Collaborative Research Center TRR 109
``Discretization in Geometry and Dynamics'' for its support.

\small

\bibliographystyle{plain}
\bibliography{Ref_Scattering}

\end{document}